\documentclass[reqno, a4paper]{amsart}
\usepackage{amsmath, amssymb, enumerate, esint, bbm, graphicx}
\usepackage[colorlinks=true,linkcolor=blue, citecolor=red]{hyperref}
\usepackage[nameinlink,capitalize]{cleveref}
\usepackage{physics}
\usepackage{tikz}
\allowdisplaybreaks

\newtheorem{theorem}{Theorem}[section]
\newtheorem{lemma}[theorem]{Lemma}

\theoremstyle{definition}
\newtheorem{definition}[theorem]{Definition}
\theoremstyle{remark}

\numberwithin{equation}{section}

\newcommand{\R}{\mathbb{R}}
\newcommand{\C}{\mathbb{C}}

\newcommand{\N}{\mathbb{N}}
\newcommand{\Z}{\mathbb{Z}}

\newcommand{\D}{\mathbb{D}}
\newcommand{\dz}{\;\mathrm{d}z}
\newcommand{\daz}{\;|\mathrm{d}z|}
\newcommand{\domega}{\;\mathrm{d}\omega}
\newcommand{\pd}{\partial}
\renewcommand{\H}{\mathbb{H}}

\let\temp\phi
\let\phi\varphi
\let\varphi\temp

\newcommand{\az}{\alpha}
\newcommand{\bz}{\beta}

\newcommand{\vz}{\phi}
\newcommand{\Oz}{\Omega}

\newcommand{\tz}{\theta}

\newcommand{\gz}{\gamma}

\newcommand{\wt}{\widetilde}
\newcommand{\bN}{\mathbb N}
\newcommand{\fz}{\infty}
\newcommand{\od}{\mathrm{d}}
\title[Homeomorphic Sobolev extensions]{Homeomorphic Sobolev extensions of parametrizations of Jordan curves}
\author[O. Bouchala]{Ondr\v{e}j Bouchala}
\address{O. Bouchala, Czech Technical University in Prague, Faculty of Information Technology, Th\'akurova 9, 160 00 Prague 6, Czech Republic}
\email{ondrej.bouchala@gmail.com}

\author[J. J\"a\"askel\"ainen]{Jarmo J\"a\"askel\"ainen}
\address{J. J\"a\"askel\"ainen, Department of Economics and Management, Koetilantie 5, P.O. Box 28, 00014 University of Helsinki, Finland}
\email{jarmo.jaaskelainen@helsinki.fi}

\author[P. Koskela]{Pekka Koskela}
\address{P. Koskela, Department of Mathematics and Statistics, University of Jyväskylä, P.O. Box 35 (MaD), FI-40014, Finland}
\email{pekka.j.koskela@jyu.fi}

\author[H. Xu]{Haiqing Xu}
\address{H. Xu, Frontiers Science Center for Nonlinear Expectations (Ministry of Education of China), Research Center for Mathematics and Interdisciplinary Sciences, Shandong University, Qingdao 266237, China}
\email{haiqing.xu@email.sdu.edu.cn}

\author[X. Zhou]{Xilin Zhou}
\address{X. Zhou, Department of Mathematics and Statistics, University of Jyväskylä, P.O. Box 35 (MaD), FI-40014, Finland}
\email{xilin.j.zhou@jyu.fi}

\subjclass[2020]{Primary: 46E35 secondary: 30C62, 58E20.}
\keywords{Sobolev homeomorphism, Sobolev extension, hyperbolic metric}
\thanks{This research was partially supported by the
Finnish centre of excellence in Randomness and Structures of the Academy of Finland, funding decision number: 346305. H. Xu was partially supported by the Young Scientist Program of the Ministry of Science and Technology of China (No.~2021YFA1002200), the National Natural Science Foundation of China (No.~12101362 and No.~12201349), and the Natural Science Foundation of Shandong Province (No.~ZR2022YQ01 and No.~ZR2021QA032)}

\begin{document}

\begin{abstract}
    Each homeomorphic parametrization of a Jordan curve via the unit circle extends to a homeomorphism of the entire plane. It is a natural question to ask if such a homeomorphism can be chosen so as to have some Sobolev regularity.

This prompts the simplified question: for a homeomorphic embedding of the unit circle into the plane, when can we find a homeomorphism from the unit disk that has the same boundary values and integrable first-order distributional derivatives?

We give the optimal geometric criterion for the  interior Jordan domain  
so that there exists a Sobolev homeomorphic extension for any homeomorphic parametrization of the Jordan curve.

The  problem is partially motivated by trying to understand which boundary values can correspond to deformations of finite energy.
\end{abstract}
\maketitle
\section{Introduction}

If $\phi \colon \pd \D \to \C$ is a (homeomorphic) embedding,  then the Jordan curve $\phi(\pd \D)$ bounds an interior Jordan domain $\Omega$, and we know that $\phi$ extends to a homeomorphism from $\overline{\D}$ to $\overline{\Omega}$, by the Jordan-Schoenflies theorem. 
The extension given by this theorem does not need to have any regularity even when both $\phi$ and $\partial \Omega$ are regular. 

Throughout this note we will denote by $\Omega \subset \C$  the bounded interior Jordan domain corresponding to the Jordan curve $\phi(\pd \D)$. We  study the Sobolev regularity of extensions of parametrizations of $\phi(\pd \D)$ and we give a geometric criterion which guarantees that an arbitrary boundary
homeomorphism $\phi \colon \pd \D \to \pd \Omega$ admits a homeomorphic extension from $\overline{\D}$ to $\overline{\Omega}$ in the Sobolev class $W^{1,p}(\D, \C)$ (Theorem \ref{The1}). Further, our criterion is optimal (Theorem \ref{counterex}). 

To begin with a classical extension, recall that,
if $\Omega$ is convex, then the complex-valued Poisson extension of the boundary homeomorphism $\phi$ is a diffeomorphism, by the classical Rad\'o-Kneser-Choquet theorem. Moreover, one knows precisely when a harmonic function with $p$-integrable partial derivatives, $1<p\le 2,$ can be found: one simply checks if the boundary value function
belongs to the trace space of the Sobolev space $W^{1,p}(\D, \C)$.

Note though, that even when the image Jordan curve is the unit circle, one can find a boundary homeomorphism $\phi \colon \pd \D \to \pd \D$ for which one cannot obtain a homeomorphic $W^{1,2}$-extension, as was shown by Verchota \cite{Verchota2007}. This means that we have to deal with $W^{1, p}$-extensions, where $1 \leq p < 2$.

 Actually, it has turned out that, for an arbitrary Jordan domain $\Omega$, one cannot necessarily even have an extension in the class $W^{1,1}(\D, \C)$, see Zhang \cite{Zhang2019}. 

 \begin{theorem}{\cite[Theorem 1.2]{Zhang2019}}
     There is a homeomorphic parametrization $\phi \colon \pd \D \to \pd \Omega$ of a Jordan curve that does not admit a homeomorphic $W^{1,1}$-extension.
 \end{theorem}

In the above theorem $\phi$ can actually be chosen to be the boundary value of a an exterior conformal map. This calls for a geometric condition of the Jordan domain $\Omega$ that guarantees the existence of a homeomorphic $W^{1,p}$-extension. We give one in terms of the hyperbolic distance $h_{\Omega}$ of the Jordan domain $\Omega$.

\begin{theorem}\label{The1}
	Let $\Omega$ be a Jordan domain,  $z_0\in\Omega$ and $q\in(1,\infty)$. Let us assume that
\begin{equation}\label{Lq-integroituva}
\int_{\Omega}\left(h_{\Omega}(z,z_0)\right)^q\dz<\infty.
\end{equation}
Then each homeomorphic parametrization $\phi \colon \pd \D \to \pd \Omega$ of our Jordan curve has a homeomorphic extension  in the class $W^{1,p}(\D,\C)$ for all $p\in[1,2)$.
\end{theorem}

Thus we have found a sufficient geometric condition on $\Omega$ that guarantees the optimal extendability.
Our criterion \eqref{Lq-integroituva} is an intrinsic geometric criterion, since the hyperbolic distance $h_{\Omega}$ is comparable, by the Koebe distortion theorem, to the quasihyperbolic distance that is given by a path metric with the density $d(z, \pd \Omega)^{-1}$.

Our result is actually optimal. Indeed, if $q = 1$, we do not necessarily even have a $W^{1,1}$-extension. This  is shown in the next theorem.

\begin{theorem}\label{counterex}
There is a homeomorphic parametrization  $\phi: \pd\D \to \pd\Omega$ of a Jordan curve so that $\phi$ does not have a homeomorphic $W^{1,1}$-extension, even though
    $$\int_{\Omega} h_{\Omega}(z,z_0)\,\dz <\infty,
    $$
for some $z_0\in\Omega$.
\end{theorem}

The regularity of the extensions in Theorem \ref{counterex} fails badly. Namely, we show that there is a set $C$ of positive length on $\pd\D$ such that, for every homeomorphic extension $\Phi\colon \overline{\D}\to\overline{\Oz}$ and any curve $\gamma_z$ connecting an interior point of the unit disk to a point $z$ in $C,$
we have that the Euclidean length of the image curve $\Phi(\gamma_z)$ is infinite.

Previous assumptions on the Jordan domain $\Omega$ that give positive results for the homeomorphic $W^{1,p}$-extendability, 
for $p \in [1,2)$, include, for instance, the rectifiability of the Jordan curve (Koski and Onninen \cite{KoskiOnninen2021}) and the assumption that $\Omega$ is a so-called John disk (Koskela, Koski and Onninen \cite{KoKoOn20}). 
In the latter work, it was speculated that, instead of a John condition, already uniform H\"older continuity of the Riemann mapping  
should suffice. 
This expectation is confirmed by  Theorem \ref{The1}: under a uniform H\"older continuity assumption, $h_{\Omega}(\cdot,z_0)$ is integrable to any power and even exponentially integrable; in fact even the Minkowski dimension of the boundary $\pd \Omega$ is strictly smaller than two under this assumption \cite{JM,KR,SS}. Our integrability assumption \eqref{Lq-integroituva} does not forbid the boundary to have  full dimension two.

The study of Sobolev homeomorphic extensions is in the heart of variational methods in the Geometric Function Theory \cite{AIM2009, HenclKoskela2014, IwaniecMartin2001, Reshetnyak1989} and the Nonlinear Elasticity \cite{Antman1995, Ball1976,
Ciarlet1988}. In the Nonlinear Elasticity, 
the question of a Sobolev homeomorphic extension is motivated by the boundary value problems for elastic deformations. The deformation is a minimizer, where one takes the infimum of a given elastic energy functional among Sobolev
homeomorphisms. Hence one wishes to know when the competing class of admissible homeomorphisms is
nonempty, i.e., when there is a $W^{1,p}$-extension.  Theorems \ref{The1} and \ref{counterex} give the optimal geometric criterion for target shapes that can correspond to deformations of finite energy.

\section{Sobolev homeomorphic extension, Theorem \ref{The1}}
To define our homeomorphic extension we will use the general extension theorem of Koski and Onninen,  \cite[Theorem 3.1]{KO2023}. 
To this end, we recall the definition of a dyadic family of closed arcs in $\pd \D$.  We denote $\bN := \{0, 1, 2, \dots\}$ and $\bN^+ := \bN\setminus\{0\}$. 
\begin{definition}
	Let $n_0\in\N^+$. A family of closed arcs 
	$$I := \{I_{n,j}\subset\pd \D : n\geq n_0,\ j = 1,2,3,\ldots,2^n\}$$
	is called dyadic, if for every $n\geq n_0$,  arcs $I_{n,j}$, $j = 1,2,3,\ldots,2^n$, are
	\begin{enumerate}[(i)]
		\item of equal length,
		\item pairwise disjoint apart from their endpoints,
		\item covering $\pd \D$,
	\end{enumerate}
	and, for each arc $I_{n,j}$ there are two arcs in $I$ of half the length of $I_{n,j}$ 
	and so that their union is exactly $I_{n,j}$, 
	these arcs are called the children of $I_{n,j}$ and $I_{n,j}$ is their parent.
\end{definition}

\begin{theorem}{\cite[Theorem 3.1]{KO2023}}\label{t0}
	Suppose that $\Omega$ is a Jordan domain and $\phi\colon\pd \D \to \partial\Omega$ is a boundary homeomorphism.
	Suppose that for some $n_0\in\N^+$ there is a dyadic family $I = \{I_{n,j}\}$ of closed arcs in $\pd \D$ such that the following hold:
	\begin{enumerate}[1.]
		\item For each $I_{n,j}$ with $n \geq n_0$ there exists a crosscut $\Gamma_{n,j}$ (i.e., a curve in $\Omega$)   connecting 
				   the endpoints of the boundary arc $\phi(I_{n,j})\subset\partial\Omega$ and such that the estimate
				   \begin{equation}\label{e0}
					\sum^{\infty}_{n = n_0}2^{(p-2)n}\sum^{2^n}_{j = 1}\left(\ell(\Gamma_{n,j})\right)^p<\infty
				   \end{equation}
				   holds. Here $\ell$ stands for the Euclidean length. 
		\item The crosscuts $\Gamma_{n,j}$ for $n\geq n_0$ are all pairwise disjoint apart from their endpoints at the boundary, where $n$ and $j$ are allowed to range over all their possible values.
	\end{enumerate}
	Then $\phi$ admits a homeomorphic extension from $\overline{\D}$ to $\overline{\Omega}$ in the class $W^{1,p}(\D, \C)$.
\end{theorem}

As we need to sum up lengths of crosscuts $\Gamma_{n,j}$ in \eqref{e0}, 
the key is to have a good upper bound for the tail of the series, i.e., for croscuts corresponding ``short'' dyadic intervals. 
We will choose as our crosscuts conformal images of certain hyperbolic geodesics. 
For the approximation we use, for instance, the following Koebe distortion theorem (see, e.g.,  \cite[Theorem 2.10.8]{AIM2009}).
\begin{theorem}\label{t1.1}
	Suppose that $f\colon\D\to f(\D)$ is a conformal map. Then for all $z,w\in\D$ we have
	\begin{equation*}
		e^{-3h_{\D}(z,w)}\leq \frac{|f'(z)|}{|f'(w)|}\leq  e^{3h_{\D}(z,w)}.
	\end{equation*}
\end{theorem}

The following key lemma gives the needed upper bound for a crosscut.
\begin{lemma}\label{l2.1}
	Suppose that $\Omega$ is a Jordan domain and $f\colon\overline{\D}\to\overline{\Omega}$ is a homeomorphism that is conformal on $\D$. Let us assume, further, that $q\in(1,\infty)$ and
	$$
\int_{\Omega}\left(h_{\Omega}(z,f(0))\right)^q\dz<\infty.
 $$ 
 
 Let $\xi_1,\xi_2\in\pd \D$ be such that $|\xi_1 - \xi_2| \leq \frac{4 \pi}{1 + \pi^2} \approx 1.156$,
 $\gamma$ a hyperbolic geodesic connecting $\xi_1$ and $\xi_2$,
 and $\wt{\gz}$ the shorter-length part of the boundary connecting $\xi_1$ and $\xi_2$.
	Then we have, for the crosscut $\Gamma = f(\gamma)$ joining $f(\xi_1)$ and $f(\xi_2)$, that
 $$
 		\big(\ell(\Gamma)\big)^2\leq c(q)\int_{\Delta} \left(h_{\Omega}(z, f(0))\right)^q\dz,
 $$
 	where $\Delta$ is the region bounded by $\Gamma$ and $f(\wt{\gz})$.
\end{lemma}

\begin{proof} For a couple of hyperbolic distance estimates we first map the unit disk to the upper half plane via a Möbius transformation.

	Without loss of generality we may assume that our points of the unit circle are such that $\xi_1 = \overline{\xi_2}$, $\Re \xi_1 > 0$ and $\Im \xi_1 > 0$.
	We consider the Möbius transformation $T$ from the unit disk to the upper half plane $\H^+,$
	$$T(z) = a\,\frac{1-z}{1+z},$$
	where $a = \frac{1+\xi_1}{1-\xi_1}=\frac{i\Im\xi_1}{1-\Re\xi_1}$. Then $T(\xi_1) = 1 = -T(\xi_2)$, $T(1) = 0$ and $T(-1) = \infty$.
It is clear that $T(\D) = \H^+$ and that $T$ maps the hyperbolic geodesic $\gamma$ connecting $\xi_1$ and $\xi_2$ to the upper half of the unit circle, i.e., $\pd \D\cap \H^+$. 

	Moreover, we have $\Re T(0) = 0$. By our assumption,  we have that $|\Im \xi_1 - \Im \xi_2| = |\xi_1 - \xi_2| \leq  \frac{4 \pi}{1 + \pi^2}$ and thus, further, $\Im \xi_1 \leq \frac{2 \pi}{1 + \pi^2}$. Hence
	\begin{equation}\label{e2-1}
		\Im T(0) = |a| = \frac{\Im\xi_1}{1-\Re\xi_1}\geq \inf_{y \in \left(0, 2 \pi /(1 + \pi^2)\right)}\frac{y}{1-\sqrt{(1- y^2)}} = \pi.
	\end{equation}
 
Our goal is to approximate the length of the crosscut $\Gamma = f(\gamma) = f \circ T^{-1}(\pd \D \cap \H^+)$. 
To this end, we divide $\pd \D \cap \H^+$ into smaller arcs $C_m$ and approximate lengths with the help of the Koebe distortion theorem. 
Given $m\in\N^+$ we define $\theta_m := \frac{\pi}{2^m}$ and we set $C_m := \{e^{i\theta}: \theta\in (\theta_{m+1},\theta_m)\}$.
Further, for approximating purposes we define $R_m := 1 - \frac{1}{2^{m + 1}}$ and $z_m := e^{i\theta_m}$ and let
	$$A_m := \left\{re^{i\theta}: r\in (R_m,1]\text{ and }\theta \in (\theta_{m+1},\theta_m)\right\}.$$
	Then, for any $m\in\N^+$ and $\omega\in A_m$, we have that $\Im z \geq R_m \sin(\theta_{m + 1})$, when $z \in [\omega, z_m]$, and thus
	\begin{equation}\label{e2-2}
		\begin{aligned}
			h_{\H^+}(\omega,z_m)&\leq \int_{[\omega,z_m]}\frac{1}{\Im z}\daz\\
			& \leq \frac{(\theta_m - \theta_{m+1}) + (1 - R_m)}{R_m\sin \theta_{m+1}} = \frac{\theta_{m+1} + 2^{-m-1}}{(1 - 2^{-m-1})\sin \theta_{m+1}}\\
			&\leq \frac{2\theta_2}{\sin \theta_2} = \frac{\sqrt{2}}{2}\pi.
		\end{aligned}
		\end{equation}
To deal with angles between $\frac{\pi}{2}$ and $\pi$,	we define $A_{-m} := \left\{z\in\D: -\overline{z}\in A_m\right\}$, $z_{-m} := -\overline{z_m}$ and $C_{-m} = \{ z : -\overline{z} \in C_m \}$ for any $m\in\N^+$.
	Thus, \eqref{e2-2} holds true for all $m\in\Z\setminus\{0\}$. Note that $\{C_m: m \in \Z \setminus\{0\}\}$ covers $\pd \D \cap \H^+$ except for the countable set $\{ z_m : m \in \Z \setminus \{0\} \}$.

	By direct computation, we conclude that for any $m\in\Z\setminus\{0\}$ and $\omega\in A_m$
	\begin{align*}
		h_{\H^+}(\omega, T(0)) \geq \int_{[\Im \omega, \Im T(0)]} \frac{1}{\Im z}\daz = \log \frac{\Im T(0)}{\Im \omega} \geq \log \frac{\Im T(0)}{\sin \theta_{|m|}}.
	\end{align*}
	Because $\sin \theta_{|m|} \in (0,\theta_{|m|})$ for any $m\in\Z\setminus\{0\}$, by \eqref{e2-1} we have
	\begin{equation}\label{e2-3}
		h_{\H^+}(\omega, T(0)) \geq \log\left(\frac{\pi}{\theta_{|m|}}\right) \geq |m|\log 2.
	\end{equation}

	Let $g := f\circ T^{-1}$, then $g\colon \H^+\to\Omega$ is a conformal homeomorphism. Let $m\in\Z \setminus \{ 0\}$ and $\omega\in A_m$. 
We will use the Koebe distortion theorem \ref{t1.1} for $g$. First the theorem implies, for $\nu = T^{-1}(\omega)$ and $z = T^{-1}(z_m)$, 
$$
		e^{-3h_{\D}(\nu,z)}\leq \frac{|f'(\nu)|}{|f'(z)|}\leq  e^{3h_{\D}(\nu,z)}.
$$
As the hyperbolic metric is invariant under a conformal map, we change the hyperbolic metric from $\D$ to $\H^+$ by our Möbius transformation $T$ and we have
$$
		e^{-3h_{\H^+}(T(\nu),T(z))}\leq \frac{|f'(\nu)|}{|f'(z)|}\leq  e^{3h_{\H^+}(T(\nu),T(z))}.
$$
Note
$$
\frac{|f'(\nu)|}{|f'(z)|} = \frac{|g'(\omega)|}{|g'(z_m)|}\frac{|(T^{-1})'(z_m)|}{|(T^{-1})'(\omega)|} = \frac{|g'(\omega)|}{|g'(z_m)|}\frac{|\omega + a|^2}{|z_m + a|^2}
$$
and recall that $a$ is a purely imaginary number and $\omega \in A_m$. Hence, 
$$\frac{R_{|m|}}{|R_{|m|} e^{i\theta_{|m|+1}} + a|} < 1,$$
and, by triangle inequality,
\begin{align*}
    |e^{i\theta_{|m|}} + a| 
    & \leq |R_{|m|} e^{i\theta_{|m|+1}} + a| + (1 - R_{|m|}) + (\theta_{|m|} - \theta_{|m| + 1})\\
    & = \frac{|R_{|m|} e^{i\theta_{|m|+1}} + a|}{R_{|m|}}\left[R_{|m|} + (1 - R_{|m|} + \theta_{|m| + 1})\frac{R_{|m|}}{|R_{|m|} e^{i\theta_{|m|+1}} + a| }\right]\\
    & \leq |R_{|m|} e^{i\theta_{|m|+1}} + a| \frac{1 + \tz_{|m| + 1}}{R_{|m|}}.
\end{align*}
Then we can conclude that
$$
\frac{1}{6} 
< \frac{R_1^2}{(1 + \tz_2)^2}
\leq \frac{R_{|m|}^2}{(1 + \tz_{|m|+1})^2} 
\leq \frac{|R_{|m|} e^{i\theta_{|m|+1}} + a|^2}{|e^{i\theta_{|m|}} + a|^2} 
\leq \frac{|\omega + a|^2}{|z_m + a|^2} \leq 1.
$$

Thus combining our estimates we know that for any $m\in\Z \setminus \{ 0\}$ and $\omega\in A_m$
\begin{equation}\label{e2-4}
		e^{-3c_0}|g'(z_m)|\leq |g'(\omega)|\leq 6e^{3c_0}|g'(z_m)|,
	\end{equation}
	where $c_0 = \frac{\sqrt{2}}{2}\pi$, by \eqref{e2-2}. 
 
 For any $m\in\Z \setminus \{0\}$ we define 
	$B_m := g(A_m)$, 
	$\omega_m := g(z_m)$ and
	$l_m := g(C_m)$. 
 We are approximating the Euclidean length of the crosscut $\Gamma = f(\gamma) = f \circ T^{-1}(\pd \D \cap \H^+) = g(\pd \D \cap \H^+)$, i.e., $\ell(\Gamma) = \ell (g(\pd\D\cap\H^+)) 
		 = \sum_{m\in\Z\setminus\{0\}} \ell(g(C_m)) = \sum_{m\in\Z\setminus\{0\}} \ell(l_m)$.
 Now,
 for any $m\in\Z \setminus \{ 0 \}$
\begin{equation}\label{e2-5}
		\ell(l_m) = \int_{C_m}|g'(z)|\daz\leq \frac{3\pi e^{3c_0}}{2^{|m|}}|g'(z_m)|.
	\end{equation}

To get upper bound in terms of $L^q$-integral of $h_{\Omega}$, we need to approximate $|g'(z_m)|$.
By our Koebe distortion estimate \eqref{e2-4}, we conclude that
	\begin{align*}
		|B_m| = |g(A_m)| 
		& = \int_{A_m}|g'(z)|^2\dz\\
		& \geq |A_m| \,e^{-6c_0}|g'(z_m)|^2\\[0.1cm]
		& = \frac{\theta_{|m|+1}}{2}\left(1 - R_{|m|}^2\right)e^{-6c_0}|g'(z_m)|^2\\[0.1cm]
		& \geq \frac{\pi}{2^{2|m| + 3}}\,e^{-6c_0}\,|g'(z_m)|^2.
	\end{align*}
	Here we denote, in addition of the modulus, also the two-dimensional Lebesgue measure with $|\cdot|$, when talking about sets.
We estimate $|B_m|$, by \eqref{e2-3}, that yields
	\begin{align*}
		\int_{B_m} \left(h_{\Omega}(\omega, f(0))\right)^q\domega
		& = \int_{B_m} \left(h_{g^{-1}(\Omega)}(g^{-1}(\omega), g^{-1}(f(0)))\right)^q\domega\\
		& = \int_{B_m} \left(h_{\H^+}(g^{-1}(\omega), T(0))\right)^q\domega \geq |B_m| \,|m|^q \left(\log 2\right)^q.
	\end{align*}	
	Thus
	\begin{equation*}
		\int_{B_m} \left(h_{\Omega}(\omega, f(0))\right)^q\domega \geq \frac{\pi e^{-6c_0}(\log 2)^q}{2^{2|m|+3}} |g'(z_m)|^2 |m|^q.
	\end{equation*}
	Together with \eqref{e2-5}, it gives that for any $m\in\Z \setminus \{0 \}$
\begin{equation}\label{2e-6}
		\ell(l_m) \leq c_1 |m|^{-\frac{q}2}\left(\int_{B_m} \left(h_{\Omega}(\omega, f(0))\right)^q\domega\right)^{\frac12},
	\end{equation}
	where $c_1 = \left[\frac{72\pi e^{6\sqrt{2}\pi}}{(\log 2)^q}\right]^{1/2}$. 
	
Recall that $B_m = f\circ T^{-1}(A_m)$ and $\bigcup_{m \in \Z \setminus \{ 0\}}A_m \subset \overline{\D} \cap \H^+$, then 
$$\bigcup_{m \in \Z \setminus \{ 0\} }B_m\subset f\circ T^{-1}(\overline{\D} \cap \H^+) = \left( \Delta\cup \bigcup_{m \in \Z \setminus \{ 0\} }f\circ T^{-1}(C_m)\right),$$ 
where $\Delta$ is the region bounded by $\Gamma$ and $f\circ T^{-1} (\{(x,0): -1\leq x\leq 1\})$.
Noticing that $B_m$s are pairwise disjoint and $\left|\bigcup_{m \in \Z \setminus \{ 0\} }f\circ T^{-1}(C_m)\right| = 0$, we conclude, by the Cauchy-Schwarz inequality, 
	\begin{align*}
		\ell (\Gamma) = \ell (g(\pd \D\cap\H^+)) 
		& = \sum_{m\in\Z\setminus\{0\}} \ell(g(C_m)) = \sum_{m\in\Z\setminus\{0\}} \ell(l_m)\\
		& \leq 2c_1\sum_{m = 1}^{\infty} m^{-\frac{q}2}\left(\int_{B_m} \left(h_{\Omega}(\omega, f(0))\right)^q\domega\right)^{\frac12}\\
		& \leq 2c_1\left(\sum_{m = 1}^{\infty}m^{-q}\right)^{1/2} \left(\sum_{m = 1}^{\infty}\left(\int_{B_m} \left(h_{\Omega}(\omega, f(0))\right)^q\domega\right)\right)^{\frac12}\\
		& \leq 2c_1\left(\zeta(q)\right)^{1/2} \left(\int_{\Delta} \left(h_{\Omega}(\omega, f(0))\right)^q\domega\right)^{\frac12},
	\end{align*}
 where $\zeta$ is the Riemann zeta function. Our assumption is that  $q \in (1, \infty)$ and it is known that in this situation $\zeta(q)<\infty$ (and $\lim\limits_{q\to 1^+} \zeta(q) = \infty$). 
\end{proof}

\begin{definition}
	Let $k\in\N$. A sequence of points $\{x_0,x_1,\ldots,x_k\}\subset\pd \D$ is called a cycle if there exist
	$0 < \theta_1 <\ldots<\theta_k<2\pi$ such that $x_j = e^{i\theta_j}x_0$ for $j=1,...,k$.
\end{definition}

\begin{proof}[Proof of Theorem \ref{The1}]
We will use the general extension theorem \ref{t0} of Koski and Onninen,  \cite[Theorem 3.1]{KO2023} and make the necessary approximation of the lengths of crosscuts by Lemma \ref{l2.1}. 

As $\Omega$ is a Jordan domain, by the Riemann mapping theorem, there is a conformal homeomorphism $f\colon\D\to\Omega$ such that $f(0) = z_0$. The map $f$ can be extended homeomorphically to $\overline{\D}$ so that $f(\pd \D) = \partial\Omega$, by the Carath\'eodory-Osgood theorem. We use the same notation $f$ for the extended homeomorphism.

Recall that we want to extend $\phi \colon \pd\D \to \pd \Omega$  in the class $W^{1,p}(\D)$.
Fix $k\in\N$ and consider the cycle $\{x_j\}^{k}_{j = 0}\subset \pd\D$. We denote $x_{k+1} = x_0$ and $\xi_j = f^{-1}(\phi(x_j))$ for any $j = 0,\ldots,k + 1$. 
 Notice that $\xi_j \in \pd\D$ and for any $|\xi_{j + 1} - \xi_j| \leq \frac{4 \pi}{1 + \pi^2}$, $j = 1, \ldots, k$, we have a upper bound for the length of the crosscut, by Lemma \ref{l2.1}. This approximation is needed to show the convergence of the series \eqref{e0}.
 
To use Lemma \ref{l2.1} 
 we need to find a dyadic family $I := \{I_{n,j}\subset\pd \D : n\geq n_0,\ j = 1,2,3,\ldots,2^n\}$ such that the endpoints $x_j$ of dyadic intervals $I_{n, j}$ have the aforementioned property that, for $\xi_j = f^{-1}(\phi(x_j))$,
 $|\xi_{j + 1} - \xi_j| \leq \frac{4 \pi}{1 + \pi^2}$. 
 
 We first choose the starting level $n_0$ of our dyadic family $I$.
Let $N\in\N$ and $\{P_0,\ldots,P_{N}\}$ be some cycle on $\pd \D$ such that $|P_{j + 1} - P_{j}| \leq \frac{2 \pi}{1 + \pi^2}$  for all $j = 0,\ldots, N$.
Then there is a corresponding point set $\{y_0,\ldots,y_{N}\}$ on $\pd \D$ such that $y_j = \phi^{-1}(f(P_j))$. Because dyadic points are dense in $\pd \D$, we can find the smallest $n_0\in\N$, 
such that there is at most one $y_j$ in each $n_0$-level dyadic arc. 
For any $n\geq n_0$,
denote  $n$-level dyadic points by $\{x_1,x_2,\ldots,x_{2^n}\}$. By the choice of $n_0$, we have
$$\max_{j = 1,2,\ldots,2^n}|f^{-1}(\phi(x_{j+1})) - f^{-1}(\phi(x_{j}))| \leq 2\max_{j = 1,\ldots,N}|P_{j+1} - P_{j}|\leq \frac{4 \pi}{1 + \pi^2}.$$

Now, on level $n$, we take the crosscut  $\Gamma_{n, j} = f(\gamma_{n, j})$ joining $f(\xi_{j}) = \phi(x_j)$ and $f(\xi_{j + 1}) = \phi(x_{j + 1})$. Here $\gamma_{n, j}$ is a hyperbolic geodesic connecting $\xi_{j + 1} = f^{-1}(\phi(x_{j+1}))$ and $\xi_{j} = f^{-1}(\phi(x_j))$. By Lemma \ref{l2.1},
 $$
 		\sum_{j = 1}^{2^n}\big(\ell(\Gamma_{n,  j})\big)^2\leq C(q)\sum_{j = 1}^{2^n}\int_{\Delta_{n, j}} \left(h_{\Omega}(z, z_0)\right)^q\dz \leq C(q)\int_{\Omega} \left(h_{\Omega}(z, z_0)\right)^q\dz =: M,
 $$
 	where $\Delta_{n, j}$ is the region defined in Lemma \ref{l2.1} with respect to the points $\xi_{j}$ and $\xi_{j + 1}$.

Then by H\"older's inequality for $p\in [1,2)$ we have
\begin{align*}
	\sum^{\infty}_{n = n_0}2^{(p-2)n}\sum^{2^n}_{j = 1}(\ell(\Gamma_{n,j}))^p
	&\leq \sum^{\infty}_{n = n_0}2^{(p-2)n} 2^{n(1 - \frac{p}2)}\left(\sum^{2^n}_{j = 1}(\ell(\Gamma_{n,j}))^2\right)^{\frac{p}2}\\
	&\leq \sum^{\infty}_{n = n_0} 2^{n(\frac{p}2 - 1)} M^{\frac{p}2} <\infty
\end{align*}
and \eqref{e0} is satisfied. 
As homeomorphic images of geodesics, our crosscuts are pairwise disjoint apart from their endpoints, and thus, by Theorem \ref{t0}, we have shown that the boundary homeomorphism $\phi\colon\pd \D\to\partial\Omega$ admits a homeomorphic extension from $\overline{\D}$ to $\overline{\Omega}$ in the class $W^{1,p}(\D, \C)$.
\end{proof}

\section{Counterexample, Theorem \ref{counterex}}

 In this section, we prove Theorem \ref{counterex}. 
 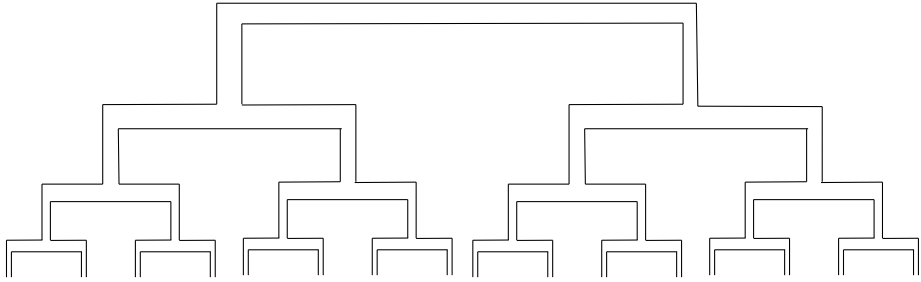
\begin{figure}[h]
 \begin{tikzpicture}[x=0.75pt,y=0.75pt,yscale=-1,xscale=1]

\draw    (210.89,103.44) -- (450.2,103.44) ;
\draw    (223.44,113.83) -- (443.56,113.44) ;
\draw    (210.89,103.44) -- (210.89,154.24) ;
\draw    (223.44,113.83) -- (223.44,154.24) ;
\draw    (154.01,154.44) -- (210.89,154.24) ;
\draw    (161.79,166.58) -- (273.28,166.58) ;
\draw    (154.01,154.44) -- (154.01,194.44) ;
\draw    (272.54,166.58) -- (272.54,193.24) ;
\draw    (280.32,154.44) -- (280.32,193.24) ;
\draw    (223.44,154.64) -- (280.32,154.44) ;
\draw    (127.9,203.24) -- (188.01,203.24) ;
\draw    (123.71,194.44) -- (123.71,222.57) ;
\draw    (127.9,203.24) -- (127.9,222.57) ;
\draw    (188.01,203.24) -- (188.01,222.57) ;
\draw    (192.21,194.44) -- (192.21,222.57) ;
\draw    (106,222.67) -- (123.94,222.57) ;
\draw    (108.45,228.48) -- (143.62,228.48) ;
\draw    (106,222.67) -- (106,241.24) ;
\draw    (108.45,228.48) -- (108.45,241.24) ;
\draw    (143.39,228.48) -- (143.39,241.24) ;
\draw    (145.84,222.67) -- (145.84,241.24) ;
\draw    (127.9,222.77) -- (145.84,222.67) ;
\draw    (170.31,222.67) -- (188.25,222.57) ;
\draw    (172.76,228.48) -- (207.93,228.48) ;
\draw    (170.31,222.67) -- (170.31,241.24) ;
\draw    (172.76,228.48) -- (172.76,241.24) ;
\draw    (207.69,228.48) -- (207.69,241.24) ;
\draw    (210.15,222.67) -- (210.15,241.24) ;
\draw    (192.21,222.77) -- (210.15,222.67) ;
\draw    (241.89,193.44) -- (272.54,193.24) ;
\draw    (246.08,202.24) -- (306.2,202.24) ;
\draw    (241.89,193.44) -- (241.89,221.57) ;
\draw    (246.08,202.24) -- (246.08,221.57) ;
\draw    (306.2,202.24) -- (306.2,221.57) ;
\draw    (310.39,193.44) -- (310.39,221.57) ;
\draw    (224.18,221.67) -- (242.12,221.57) ;
\draw    (226.64,227.48) -- (261.8,227.48) ;
\draw    (224.18,221.67) -- (224.18,240.24) ;
\draw    (226.64,227.48) -- (226.64,240.24) ;
\draw    (261.57,227.48) -- (261.57,240.24) ;
\draw    (264.02,221.67) -- (264.02,240.24) ;
\draw    (246.08,221.77) -- (264.02,221.67) ;
\draw    (288.49,221.67) -- (306.43,221.57) ;
\draw    (290.94,227.48) -- (326.11,227.48) ;
\draw    (288.49,221.67) -- (288.49,240.24) ;
\draw    (290.94,227.48) -- (290.94,240.24) ;
\draw    (325.88,227.48) -- (325.88,240.24) ;
\draw    (328.33,221.67) -- (328.33,240.24) ;
\draw    (310.39,221.77) -- (328.33,221.67) ;
\draw    (279.74,193.64) -- (310.39,193.44) ;
\draw    (123.71,194.44) -- (154.01,194.44) ;
\draw    (161.79,166.58) -- (162.14,194.44) ;
\draw    (162.14,194.44) -- (192.21,194.44) ;
\draw    (443.56,113.44) -- (443.56,154.24) ;
\draw    (450.2,103.44) -- (450.94,155.24) ;
\draw    (386.68,154.44) -- (443.56,154.24) ;
\draw    (394.46,166.58) -- (505.95,166.58) ;
\draw    (386.68,154.44) -- (386.68,194.44) ;
\draw    (505.21,166.58) -- (505.21,193.24) ;
\draw    (513,155.44) -- (512.99,193.24) ;
\draw    (450.94,155.24) -- (513,155.44) ;
\draw    (360.57,203.24) -- (420.68,203.24) ;
\draw    (356.38,194.44) -- (356.38,222.57) ;
\draw    (360.57,203.24) -- (360.57,222.57) ;
\draw    (420.68,203.24) -- (420.68,222.57) ;
\draw    (424.88,194.44) -- (424.88,222.57) ;
\draw    (338.67,222.67) -- (356.61,222.57) ;
\draw    (341.12,228.48) -- (376.29,228.48) ;
\draw    (338.67,222.67) -- (338.67,241.24) ;
\draw    (341.12,228.48) -- (341.12,241.24) ;
\draw    (376.06,228.48) -- (376.06,241.24) ;
\draw    (378.51,222.67) -- (378.51,241.24) ;
\draw    (360.57,222.77) -- (378.51,222.67) ;
\draw    (402.98,222.67) -- (420.92,222.57) ;
\draw    (405.43,228.48) -- (440.6,228.48) ;
\draw    (402.98,222.67) -- (402.98,241.24) ;
\draw    (405.43,228.48) -- (405.43,241.24) ;
\draw    (440.36,228.48) -- (440.36,241.24) ;
\draw    (442.82,222.67) -- (442.82,241.24) ;
\draw    (424.88,222.77) -- (442.82,222.67) ;
\draw    (474.56,193.44) -- (505.21,193.24) ;
\draw    (478.75,202.24) -- (538.87,202.24) ;
\draw    (474.56,193.44) -- (474.56,221.57) ;
\draw    (478.75,202.24) -- (478.75,221.57) ;
\draw    (538.87,202.24) -- (538.87,221.57) ;
\draw    (543.06,193.44) -- (543.06,221.57) ;
\draw    (456.85,221.67) -- (474.79,221.57) ;
\draw    (459.31,227.48) -- (494.47,227.48) ;
\draw    (456.85,221.67) -- (456.85,240.24) ;
\draw    (459.31,227.48) -- (459.31,240.24) ;
\draw    (494.24,227.48) -- (494.24,240.24) ;
\draw    (496.69,221.67) -- (496.69,240.24) ;
\draw    (478.75,221.77) -- (496.69,221.67) ;
\draw    (521.16,221.67) -- (539.1,221.57) ;
\draw    (523.61,227.48) -- (558.78,227.48) ;
\draw    (521.16,221.67) -- (521.16,240.24) ;
\draw    (523.61,227.48) -- (523.61,240.24) ;
\draw    (558.55,227.48) -- (558.55,240.24) ;
\draw    (561,221.67) -- (561,240.24) ;
\draw    (543.06,221.77) -- (561,221.67) ;
\draw    (512.41,193.64) -- (543.06,193.44) ;
\draw    (356.38,194.44) -- (386.68,194.44) ;
\draw    (394.46,166.58) -- (394.81,194.44) ;
\draw    (394.81,194.44) -- (424.88,194.44) ;

\end{tikzpicture}
\caption{Schematic picture of the Jordan domain}\label{schema}
\end{figure}

 Before the proof, we fix some notation. The symbol $A\lesssim B$ means that $A\leq c B$ for some positive constant $c$,
while $A\sim B$ means $A\lesssim B\lesssim A$.
If $f\leq cg$ and $g = h$ or $g \leq h$, we then write $f\lesssim g \sim h$ or $f\lesssim g\lesssim h$,
rather than $f\lesssim g = h$ or $f\lesssim g\leq h$.    For any $x,y\in\R^2$,  $[x,y]$  is a line segment between $x$ and $y$.
    Moreover, for given lines (or line segments) $\az,\bz\subset \R^2$,
    we write $\az \perp \bz$ if they are orthogonal to each other. For any $r\in\R$, 
$\lceil r \rceil := \min\{n\in\Z: n\geq r\}.$
 
    Let $C \subset [0, 1]$ be a Smith--Volterra--Cantor set, i.e., we remove an open interval with the 
    length $4^{-n}$ from the middle of each of the $2^{n - 1}$ intervals in $n$-step construction of $C$. 
    From the left to the right, we denote by $I_{n,1},\,\dots,\,I_{n,2^n}$ the 
    remaining closed intervals in the $n$-step. They have the length $\frac{(1 + 2^{-n})}{2^{n + 1}}$.

We construct  a tree-like Jordan domain, in which one needs to travel  infinitely long to reach a point of $C$; the points are at the ends of branches, see the schematic picture in Figure \ref{schema}. This gives the desired counterexample when the parametrization is chosen suitably.

    \subsection{Construction of  a tree-type core \texorpdfstring{$\gz$ of the Jordan domain}{Lg}\label{ss1}}
    For any $n\in\bN^+$, we first construct a tree-type structure $\gz_n\subset [0,1]\times[0,1]$ that gives the core bones of our domain.
    For any $k\in\bN^+$ such that $k\leq n$,
    consider a canonical rectangle $I_{k,j}\times [2^{-k},2^{-(k + 1)}]$, where $j = \{1,\dots,2^k\}$. 
    
    Denote by $p_{k,j}$ the middle point of the interval $I_{k,j} \times \{2^{-k}\}$,
    and by $q_{k,j}$ the middle point of the interval $I_{k,j} \times \{2^{-(k + 1)}\}$.
    Let $p_0$ be the middle point of the line segment $[p_{1,1}, p_{1,2}]$.
    Divide the line segment $[p_{k,j},q_{k,j}]$ into $2^{k + 1}$ equivalent segments. 
    The segments with ending points $p_{k,j}$ and $q_{k,j}$ are kept fixed. 
    For the rest $2^{k + 1} - 2$ segments, 
    we alternatively push them onto the left or the right side of the square $I_{k,j}\times [2^{-k},2^{-(k + 1)}]$
    in a way that any two of the new segments do not connect.
    Join each pair of these segments by horizontal line segments, we obtain a ``snake" curve $\wt{\gz}_{k,j}$ connecting $p_{k,j}$ and $q_{k,j}$. 
    Let
    $$\wt{\gz}_n = \bigcup^n_{k = 1}\bigcup^{2^{k - 1}}_{i = 1} \wt{\gz}_{k,2(i - 1)}\cup \wt{\gz}_{k,2i} \cup [p_{k,2(i - 1)},p_{k,2i}].$$
    See the figure below.

    \smallskip
    \tikzset{every picture/.style={line width=0.75pt}} 

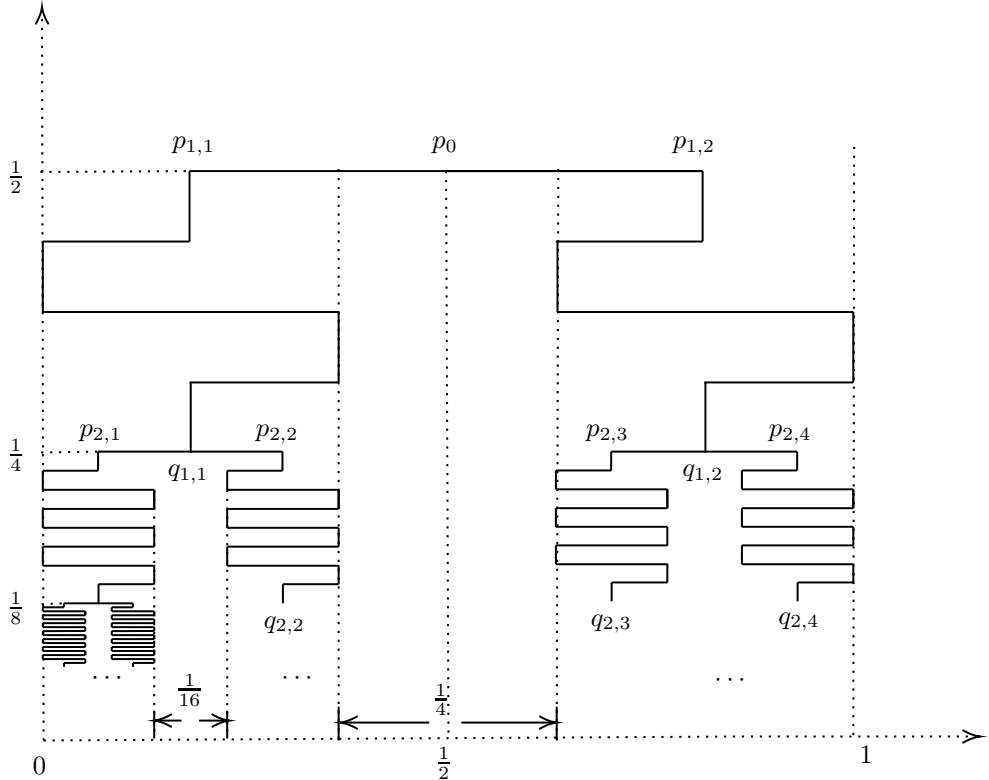
\begin{tikzpicture}[x=0.6pt,y=0.75pt,yscale=-1,xscale=1]

\draw    (310.5,134) -- (470.5,134) ;
\draw    (150.5,134) -- (150.5,169.44) ;
\draw    (59,169.44) -- (150.5,169.44) ;
\draw    (59,169.44) -- (59,204.89) ;
\draw    (59,204.89) -- (243.5,204.89) ;
\draw    (243.5,204.89) -- (243.5,240.33) ;
\draw    (150.5,240.33) -- (243.5,240.33) ;
\draw    (151.25,240.33) -- (151.25,275.77) ;
\draw    (470.5,134) -- (470.5,169.44) ;
\draw    (150.5,134) -- (470.5,134) ;
\draw    (379,169.44) -- (470.5,169.44) ;
\draw    (380,169.44) -- (380,204.89) ;
\draw    (380,204.89) -- (564.5,204.89) ;
\draw    (564.5,204.89) -- (564.5,240.33) ;
\draw    (471.5,240.33) -- (564.5,240.33) ;
\draw    (472.25,240.33) -- (472.25,275.77) ;
\draw  [dash pattern={on 0.84pt off 2.51pt}]  (310.5,134) -- (312,418.44) ;
\draw    (93.47,275.13) -- (93.47,284.74) ;
\draw    (59,284.74) -- (93.47,284.74) ;
\draw    (59,284.74) -- (59,294.35) ;
\draw    (59,294.35) -- (128.5,294.35) ;
\draw    (128.5,294.35) -- (128.5,303.96) ;
\draw    (59,303.96) -- (128.5,303.96) ;
\draw    (128.5,294.35) -- (128.5,303.96) ;
\draw    (59,303.96) -- (59,313.58) ;
\draw    (128.5,313.45) -- (128.5,323.06) ;
\draw    (59,322.93) -- (59,332.54) ;
\draw    (128.5,332.28) -- (128.5,341.89) ;
\draw    (93.75,341.83) -- (93.75,351.44) ;
\draw    (59,313.45) -- (128.5,313.45) ;
\draw    (59,332.54) -- (128.5,332.54) ;
\draw    (59,322.93) -- (128.5,322.93) ;
\draw    (93.75,341.83) -- (128.5,341.83) ;
\draw    (208.47,275.13) -- (208.47,284.74) ;
\draw    (174,284.74) -- (208.47,284.74) ;
\draw    (174,284.74) -- (174,294.35) ;
\draw    (174,294.35) -- (243.5,294.35) ;
\draw    (243.5,294.35) -- (243.5,303.96) ;
\draw    (174,303.96) -- (243.5,303.96) ;
\draw    (243.5,294.35) -- (243.5,303.96) ;
\draw    (174,303.96) -- (174,313.58) ;
\draw    (243.5,313.45) -- (243.5,323.06) ;
\draw    (174,322.93) -- (174,332.54) ;
\draw    (243.5,332.28) -- (243.5,341.89) ;
\draw    (208.75,341.83) -- (208.75,351.44) ;
\draw    (174,313.45) -- (243.5,313.45) ;
\draw    (174,332.54) -- (243.5,332.54) ;
\draw    (174,322.93) -- (243.5,322.93) ;
\draw    (208.75,341.83) -- (243.5,341.83) ;
\draw    (93.47,275.13) -- (208.47,275.13) ;
\draw    (413.47,275.13) -- (413.47,284.61) ;
\draw    (379,284.61) -- (413.47,284.61) ;
\draw    (379,284.61) -- (379,294.1) ;
\draw    (379,294.1) -- (448.5,294.1) ;
\draw    (448.5,294.1) -- (448.5,303.59) ;
\draw    (379,303.59) -- (448.5,303.59) ;
\draw    (448.5,294.1) -- (448.5,303.59) ;
\draw    (379,303.59) -- (379,313.07) ;
\draw    (448.5,312.94) -- (448.5,322.43) ;
\draw    (379,322.3) -- (379,331.79) ;
\draw    (448.5,331.53) -- (448.5,341.02) ;
\draw    (413.75,340.95) -- (413.75,350.44) ;
\draw    (379,312.94) -- (448.5,312.94) ;
\draw    (379,331.79) -- (448.5,331.79) ;
\draw    (379,322.3) -- (448.5,322.3) ;
\draw    (413.75,340.95) -- (448.5,340.95) ;
\draw    (529.47,275.13) -- (529.47,284.61) ;
\draw    (495,284.61) -- (529.47,284.61) ;
\draw    (495,284.61) -- (495,294.1) ;
\draw    (495,294.1) -- (564.5,294.1) ;
\draw    (564.5,294.1) -- (564.5,303.59) ;
\draw    (495,303.59) -- (564.5,303.59) ;
\draw    (564.5,294.1) -- (564.5,303.59) ;
\draw    (495,303.59) -- (495,313.07) ;
\draw    (564.5,312.94) -- (564.5,322.43) ;
\draw    (495,322.3) -- (495,331.79) ;
\draw    (564.5,331.53) -- (564.5,341.02) ;
\draw    (529.75,340.95) -- (529.75,350.44) ;
\draw    (495,312.94) -- (564.5,312.94) ;
\draw    (495,331.79) -- (564.5,331.79) ;
\draw    (495,322.3) -- (564.5,322.3) ;
\draw    (529.75,340.95) -- (564.5,340.95) ;
\draw    (413.47,275.13) -- (529.47,275.13) ;
\draw    (72.14,351.44) -- (72.14,353.45) ;
\draw    (59,353.45) -- (72.14,353.45) ;
\draw    (59,353.45) -- (59,355.45) ;
\draw    (59,355.45) -- (85.5,355.45) ;
\draw    (85.5,355.45) -- (85.5,357.46) ;
\draw    (59,357.46) -- (85.5,357.46) ;
\draw    (85.5,355.45) -- (85.5,357.46) ;
\draw    (59,357.46) -- (59,359.47) ;
\draw    (85.5,359.44) -- (85.5,361.45) ;
\draw    (59,361.42) -- (59,363.43) ;
\draw    (85.5,363.38) -- (85.5,365.39) ;
\draw    (59,359.44) -- (85.5,359.44) ;
\draw    (59,363.43) -- (85.5,363.43) ;
\draw    (59,361.42) -- (85.5,361.42) ;
\draw    (59,365.44) -- (85.5,365.37) ;
\draw    (59,365.44) -- (59,367.45) ;
\draw    (59,367.45) -- (85.5,367.45) ;
\draw    (85.5,367.45) -- (85.5,369.45) ;
\draw    (59,369.45) -- (85.5,369.45) ;
\draw    (85.5,367.45) -- (85.5,369.45) ;
\draw    (59,369.45) -- (59,371.46) ;
\draw    (85.5,371.44) -- (85.5,373.44) ;
\draw    (59,373.42) -- (59,375.42) ;
\draw    (85.5,375.37) -- (85.5,377.38) ;
\draw    (59,371.44) -- (85.5,371.44) ;
\draw    (59,375.42) -- (85.5,375.42) ;
\draw    (59,373.42) -- (85.5,373.42) ;
\draw    (85.5,375.5) -- (85.5,377.51) ;
\draw    (59,377.48) -- (59,379.49) ;
\draw    (85.5,379.43) -- (85.5,381.44) ;
\draw    (72.25,381.43) -- (72.25,383.44) ;
\draw    (59,379.49) -- (85.5,379.49) ;
\draw    (72.25,381.43) -- (85.5,381.43) ;
\draw    (59,377.48) -- (85.5,377.51) ;
\draw    (115.14,351.44) -- (115.14,353.45) ;
\draw    (102,353.45) -- (115.14,353.45) ;
\draw    (102,353.45) -- (102,355.45) ;
\draw    (102,355.45) -- (128.5,355.45) ;
\draw    (128.5,355.45) -- (128.5,357.46) ;
\draw    (102,357.46) -- (128.5,357.46) ;
\draw    (128.5,355.45) -- (128.5,357.46) ;
\draw    (102,357.46) -- (102,359.47) ;
\draw    (128.5,359.44) -- (128.5,361.45) ;
\draw    (102,361.42) -- (102,363.43) ;
\draw    (128.5,363.38) -- (128.5,365.39) ;
\draw    (102,359.44) -- (128.5,359.44) ;
\draw    (102,363.43) -- (128.5,363.43) ;
\draw    (102,361.42) -- (128.5,361.42) ;
\draw    (102,365.44) -- (128.5,365.37) ;
\draw    (102,365.44) -- (102,367.45) ;
\draw    (102,367.45) -- (128.5,367.45) ;
\draw    (128.5,367.45) -- (128.5,369.45) ;
\draw    (102,369.45) -- (128.5,369.45) ;
\draw    (128.5,367.45) -- (128.5,369.45) ;
\draw    (102,369.45) -- (102,371.46) ;
\draw    (128.5,371.44) -- (128.5,373.44) ;
\draw    (102,373.42) -- (102,375.42) ;
\draw    (128.5,375.37) -- (128.5,377.38) ;
\draw    (102,371.44) -- (128.5,371.44) ;
\draw    (102,375.42) -- (128.5,375.42) ;
\draw    (102,373.42) -- (128.5,373.42) ;
\draw    (128.5,375.5) -- (128.5,377.51) ;
\draw    (102,377.48) -- (102,379.49) ;
\draw    (128.5,379.43) -- (128.5,381.44) ;
\draw    (115.25,381.43) -- (115.25,383.44) ;
\draw    (102,379.49) -- (128.5,379.49) ;
\draw    (115.25,381.43) -- (128.5,381.43) ;
\draw    (102,377.48) -- (128.5,377.51) ;
\draw    (72.14,351.44) -- (115.14,351.44) ;
\draw  [dash pattern={on 0.84pt off 2.51pt}]  (243.5,133) -- (243.5,420.44) ;
\draw  [dash pattern={on 0.84pt off 2.51pt}]  (380.5,133) -- (379.5,420.44) ;
\draw  [dash pattern={on 0.84pt off 2.51pt}]  (565,121.8) -- (564.5,418.44) ;
\draw    (243.5,420.44) -- (243.5,403.44) ;
\draw    (379.5,420.44) -- (379.5,403.44) ;
\draw    (320,411.44) -- (377.5,411.44) ;
\draw [shift={(379.5,411.44)}, rotate = 180] [color={rgb, 255:red, 0; green, 0; blue, 0 }  ][line width=0.75]    (10.93,-3.29) .. controls (6.95,-1.4) and (3.31,-0.3) .. (0,0) .. controls (3.31,0.3) and (6.95,1.4) .. (10.93,3.29)   ;
\draw    (300,411.44) -- (246.5,411.44) ;
\draw [shift={(244.5,411.44)}, rotate = 360] [color={rgb, 255:red, 0; green, 0; blue, 0 }  ][line width=0.75]    (10.93,-3.29) .. controls (6.95,-1.4) and (3.31,-0.3) .. (0,0) .. controls (3.31,0.3) and (6.95,1.4) .. (10.93,3.29)   ;
\draw  [dash pattern={on 0.84pt off 2.51pt}]  (128.5,294.35) -- (128.5,419.44) ;
\draw  [dash pattern={on 0.84pt off 2.51pt}]  (174,294.35) -- (174,419.44) ;
\draw    (128.5,419.44) -- (128.5,405.44) ;
\draw    (174,419.44) -- (174,405.44) ;
\draw    (145.5,410.44) -- (130.5,410.44) ;
\draw [shift={(128.5,410.44)}, rotate = 360] [color={rgb, 255:red, 0; green, 0; blue, 0 }  ][line width=0.75]    (10.93,-3.29) .. controls (6.95,-1.4) and (3.31,-0.3) .. (0,0) .. controls (3.31,0.3) and (6.95,1.4) .. (10.93,3.29)   ;
\draw    (156.5,410.44) -- (173.25,410.44) ;
\draw [shift={(175.25,410.44)}, rotate = 180] [color={rgb, 255:red, 0; green, 0; blue, 0 }  ][line width=0.75]    (10.93,-3.29) .. controls (6.95,-1.4) and (3.31,-0.3) .. (0,0) .. controls (3.31,0.3) and (6.95,1.4) .. (10.93,3.29)   ;
\draw  [dash pattern={on 0.84pt off 2.51pt}]  (59.5,420.44) -- (494.5,418.44) -- (641.5,418.44) ;
\draw [shift={(643.5,418.44)}, rotate = 180] [color={rgb, 255:red, 0; green, 0; blue, 0 }  ][line width=0.75]    (10.93,-3.29) .. controls (6.95,-1.4) and (3.31,-0.3) .. (0,0) .. controls (3.31,0.3) and (6.95,1.4) .. (10.93,3.29)   ;
\draw  [dash pattern={on 0.84pt off 2.51pt}]  (59.5,420.44) -- (58.51,53.44) ;
\draw [shift={(58.5,51.44)}, rotate = 89.84] [color={rgb, 255:red, 0; green, 0; blue, 0 }  ][line width=0.75]    (10.93,-3.29) .. controls (6.95,-1.4) and (3.31,-0.3) .. (0,0) .. controls (3.31,0.3) and (6.95,1.4) .. (10.93,3.29)   ;
\draw  [dash pattern={on 0.84pt off 2.51pt}]  (57.5,134.44) -- (150.5,134) ;
\draw  [dash pattern={on 0.84pt off 2.51pt}]  (57.97,275.57) -- (93.47,275.13) ;
\draw  [dash pattern={on 0.84pt off 2.51pt}]  (58.25,351.88) -- (72.14,351.44) ;

\draw (300,390.44) node [anchor=north west][inner sep=0.75pt]   [align=left] {$\frac{1}{4}$};
\draw (140,385.44) node [anchor=north west][inner sep=0.75pt]   [align=left] {$\frac{1}{16}$};
\draw (51,427.44) node [anchor=north west][inner sep=0.75pt]   [align=left] {0};
\draw (566.5,421.44) node [anchor=north west][inner sep=0.75pt]   [align=left] {1};
\draw (302,422.44) node [anchor=north west][inner sep=0.75pt]   [align=left] {$\frac{1}{2}$};
\draw (35,127.44) node [anchor=north west][inner sep=0.75pt]   [align=left] {$\frac{1}{2}$};
\draw (35,267.44) node [anchor=north west][inner sep=0.75pt]   [align=left] {$\frac{1}{4}$};
\draw (87.5,384.44) node [anchor=north west][inner sep=0.75pt]   [align=left] {$\cdots$};
\draw (206,384.44) node [anchor=north west][inner sep=0.75pt]   [align=left] {$\cdots$};
\draw (476,385.44) node [anchor=north west][inner sep=0.75pt]   [align=left] {$\cdots$};
\draw (300,115) node [anchor=north west][inner sep=0.75pt]   [align=left] {$p_0$};
\draw (138,115) node [anchor=north west][inner sep=0.75pt]   [align=left] {$p_{1,1}$};
\draw (450,115) node [anchor=north west][inner sep=0.75pt]   [align=left] {$p_{1,2}$};
\draw (135,280) node [anchor=north west][inner sep=0.75pt]   [align=left] {$q_{1,1}$};
\draw (456,280) node [anchor=north west][inner sep=0.75pt]   [align=left] {$q_{1,2}$};
\draw (80,260) node [anchor=north west][inner sep=0.75pt]   [align=left] {$p_{2,1}$};
\draw (190,260) node [anchor=north west][inner sep=0.75pt]   [align=left] {$p_{2,2}$};
\draw (396,260) node [anchor=north west][inner sep=0.75pt]   [align=left] {$p_{2,3}$};
\draw (510,260) node [anchor=north west][inner sep=0.75pt]   [align=left] {$p_{2,4}$};
\draw (195,356.44) node [anchor=north west][inner sep=0.75pt]   [align=left] {$q_{2,2}$};
\draw (399,355.44) node [anchor=north west][inner sep=0.75pt]   [align=left] {$q_{2,3}$};
\draw (516,354.44) node [anchor=north west][inner sep=0.75pt]   [align=left] {$q_{2,4}$};
\draw (35,344.44) node [anchor=north west][inner sep=0.75pt]   [align=left] {$\frac{1}{8}$};

\end{tikzpicture}

\smallskip
    Note that the angle at every non-smooth point is $\pi/2$. We proceed as follows: for any $k\leq n$ and $j\in\{1,\dots,2^k\}$,
    \begin{itemize}
        \item[(1)] replace a symmetric neighborhood of one point in $\wt{\gz}_{k,j}\subset\wt{\gz}_n$ by a quarter of a circle internally tangent $\wt{\gz}_{k,j}$ with radius $2^{-2(k + 1) - 1}$.
        \item[(2)] for the point $q_{k,j}$ with $k\leq n$, make the internally tangent circles with the help from the line segments $[p_{k + 1,2j - 1},q_{k,j}]$ and $[p_{k + 1,2j},q_{k,j}]$.
    \end{itemize}
    The following figure shows the situation when $j\in\bN^+$ is odd. When $j$ is even, the construction is similar. 

    \smallskip
    \tikzset{every picture/.style={line width=0.75pt}} 

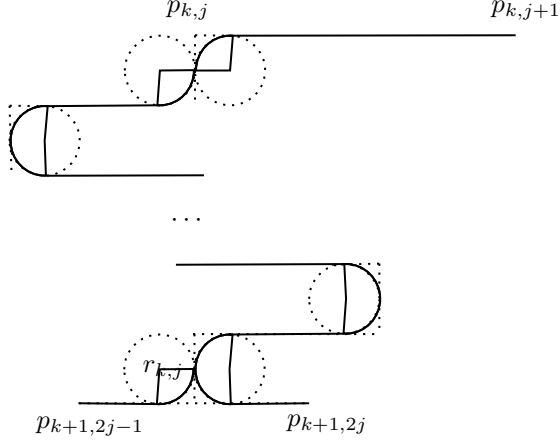
\begin{tikzpicture}[x=0.75pt,y=0.75pt,yscale=-1,xscale=1]

\draw  [dash pattern={on 0.84pt off 2.51pt}]  (150.5,134) -- (150.5,169.44) ;
\draw  [dash pattern={on 0.84pt off 2.51pt}]  (58,169.44) -- (76.5,169.66) ;
\draw  [dash pattern={on 0.84pt off 2.51pt}]  (58,169.44) -- (59,204.89) ;
\draw  [dash pattern={on 0.84pt off 2.51pt}]  (150.5,134) -- (168,134.22) ;
\draw  [dash pattern={on 0.84pt off 2.51pt}] (57.99,187.1) .. controls (57.99,177.43) and (65.83,169.6) .. (75.5,169.6) .. controls (85.16,169.6) and (93,177.43) .. (93,187.1) .. controls (93,196.77) and (85.16,204.6) .. (75.5,204.6) .. controls (65.83,204.6) and (57.99,196.77) .. (57.99,187.1) -- cycle ;
\draw  [dash pattern={on 0.84pt off 2.51pt}] (150.5,151.72) .. controls (150.5,142.05) and (158.34,134.22) .. (168,134.22) .. controls (177.67,134.22) and (185.51,142.05) .. (185.51,151.72) .. controls (185.51,161.39) and (177.67,169.22) .. (168,169.22) .. controls (158.34,169.22) and (150.5,161.39) .. (150.5,151.72) -- cycle ;
\draw  [dash pattern={on 0.84pt off 2.51pt}] (115.49,151.72) .. controls (115.49,142.05) and (123.33,134.22) .. (133,134.22) .. controls (142.66,134.22) and (150.5,142.05) .. (150.5,151.72) .. controls (150.5,161.39) and (142.66,169.22) .. (133,169.22) .. controls (123.33,169.22) and (115.49,161.39) .. (115.49,151.72) -- cycle ;
\draw    (168,134.22) -- (310.5,134) ;
\draw    (76.5,169.66) -- (133,169.22) ;
\draw    (76.5,204.67) -- (155,204.44) ;
\draw  [draw opacity=0] (76.13,204.59) .. controls (75.92,204.6) and (75.71,204.6) .. (75.5,204.6) .. controls (66.11,204.6) and (58.5,196.77) .. (58.5,187.1) .. controls (58.5,177.43) and (66.11,169.6) .. (75.5,169.6) .. controls (76.01,169.6) and (76.51,169.62) .. (77.01,169.67) -- (75.5,187.1) -- cycle ; \draw   (76.13,204.59) .. controls (75.92,204.6) and (75.71,204.6) .. (75.5,204.6) .. controls (66.11,204.6) and (58.5,196.77) .. (58.5,187.1) .. controls (58.5,177.43) and (66.11,169.6) .. (75.5,169.6) .. controls (76.01,169.6) and (76.51,169.62) .. (77.01,169.67) ;  
\draw  [draw opacity=0] (149.99,151.55) .. controls (150.09,161.19) and (142.61,169.09) .. (133.25,169.23) .. controls (132.74,169.23) and (132.23,169.22) .. (131.74,169.18) -- (133,151.72) -- cycle ; \draw   (149.99,151.55) .. controls (150.09,161.19) and (142.61,169.09) .. (133.25,169.23) .. controls (132.74,169.23) and (132.23,169.22) .. (131.74,169.18) ;  
\draw  [draw opacity=0] (151.01,151.75) .. controls (150.99,142.11) and (158.54,134.27) .. (167.9,134.21) .. controls (168.41,134.21) and (168.91,134.23) .. (169.41,134.27) -- (168,151.72) -- cycle ; \draw   (151.01,151.75) .. controls (150.99,142.11) and (158.54,134.27) .. (167.9,134.21) .. controls (168.41,134.21) and (168.91,134.23) .. (169.41,134.27) ;  
\draw  [dash pattern={on 0.84pt off 2.51pt}]  (219,249.44) -- (242.5,248.89) ;
\draw  [dash pattern={on 0.84pt off 2.51pt}]  (242.5,248.89) -- (242.5,284.33) ;
\draw  [dash pattern={on 0.84pt off 2.51pt}]  (149.5,284.33) -- (242.5,284.33) ;
\draw  [dash pattern={on 0.84pt off 2.51pt}]  (150.25,284.33) -- (150.25,319.77) ;
\draw  [dash pattern={on 0.84pt off 2.51pt}]  (92.47,319.13) -- (207.47,319.13) ;
\draw  [dash pattern={on 0.84pt off 2.51pt}] (207.49,266.61) .. controls (207.49,256.94) and (215.33,249.1) .. (225,249.1) .. controls (234.66,249.1) and (242.5,256.94) .. (242.5,266.61) .. controls (242.5,276.27) and (234.66,284.11) .. (225,284.11) .. controls (215.33,284.11) and (207.49,276.27) .. (207.49,266.61) -- cycle ;
\draw  [dash pattern={on 0.84pt off 2.51pt}] (115.24,302.05) .. controls (115.24,292.38) and (123.08,284.55) .. (132.75,284.55) .. controls (142.41,284.55) and (150.25,292.38) .. (150.25,302.05) .. controls (150.25,311.72) and (142.41,319.55) .. (132.75,319.55) .. controls (123.08,319.55) and (115.24,311.72) .. (115.24,302.05) -- cycle ;
\draw  [dash pattern={on 0.84pt off 2.51pt}] (150.25,302.05) .. controls (150.25,292.38) and (158.09,284.55) .. (167.75,284.55) .. controls (177.42,284.55) and (185.26,292.38) .. (185.26,302.05) .. controls (185.26,311.72) and (177.42,319.55) .. (167.75,319.55) .. controls (158.09,319.55) and (150.25,311.72) .. (150.25,302.05) -- cycle ;
\draw    (141,249.44) -- (225,249.1) ;
\draw    (167.75,284.55) -- (225,284.11) ;
\draw    (92.47,319.13) -- (132.75,319.55) ;
\draw    (167.75,319.55) -- (207.47,319.13) ;
\draw  [draw opacity=0] (168.38,319.54) .. controls (168.18,319.55) and (167.96,319.55) .. (167.75,319.55) .. controls (158.37,319.55) and (150.76,311.72) .. (150.76,302.05) .. controls (150.76,292.38) and (158.37,284.55) .. (167.75,284.55) .. controls (168.26,284.55) and (168.77,284.57) .. (169.26,284.62) -- (167.75,302.05) -- cycle ; \draw   (168.38,319.54) .. controls (168.18,319.55) and (167.96,319.55) .. (167.75,319.55) .. controls (158.37,319.55) and (150.76,311.72) .. (150.76,302.05) .. controls (150.76,292.38) and (158.37,284.55) .. (167.75,284.55) .. controls (168.26,284.55) and (168.77,284.57) .. (169.26,284.62) ;  
\draw  [draw opacity=0] (225,249.1) .. controls (225.21,249.09) and (225.42,249.09) .. (225.63,249.08) .. controls (235.02,248.95) and (242.74,256.68) .. (242.88,266.35) .. controls (243.01,276.01) and (235.52,283.96) .. (226.13,284.09) .. controls (225.62,284.1) and (225.12,284.08) .. (224.62,284.05) -- (225.88,266.59) -- cycle ; \draw   (225,249.1) .. controls (225.21,249.09) and (225.42,249.09) .. (225.63,249.08) .. controls (235.02,248.95) and (242.74,256.68) .. (242.88,266.35) .. controls (243.01,276.01) and (235.52,283.96) .. (226.13,284.09) .. controls (225.62,284.1) and (225.12,284.08) .. (224.62,284.05) ;  
\draw  [draw opacity=0] (149.74,301.88) .. controls (149.84,311.51) and (142.36,319.42) .. (133,319.55) .. controls (132.49,319.56) and (131.98,319.55) .. (131.49,319.51) -- (132.75,302.05) -- cycle ; \draw   (149.74,301.88) .. controls (149.84,311.51) and (142.36,319.42) .. (133,319.55) .. controls (132.49,319.56) and (131.98,319.55) .. (131.49,319.51) ;  

\draw (135,115) node [anchor=north west][inner sep=0.75pt]   [align=left] {$p_{k,j}$};
\draw (297,115) node [anchor=north west][inner sep=0.75pt]   [align=left] {$p_{k,j + 1}$};
\draw (123,296) node [anchor=north west][inner sep=0.75pt]   [align=left] {$r_{k,j}$};
\draw (70,324.44) node [anchor=north west][inner sep=0.75pt]   [align=left] {$p_{k + 1,2j - 1}$};
\draw (195,323.44) node [anchor=north west][inner sep=0.75pt]   [align=left] {$p_{k + 1,2j}$};
\draw (137,222.44) node [anchor=north west][inner sep=0.75pt]   [align=left] {$\cdots$};
\end{tikzpicture}

\smallskip

    Thus, we get a smooth ``curve'' $\gz_n$. Near the original position of the point $q_{k,j}$, there is a new ``branch" point on $\gz_n$ denoted by $r_{k,j}$. Specially, let $r_{0,1} := p_0$.
    Let $\gz_{k,j}\subset \gz_n$ be the shortest smooth curve connecting $r_{k - 1,\lceil j/2\rceil}$ and $r_{k,j}$ for any $j\in\{1,\dots,2^k\}$.
    Notice that $2\geq \ell(\wt{\gz}_{k,j})\geq 1$ for any $k\in\bN$ and $j\in\{1,\dots,2^k\}$.
    By slightly modification, we can further let
\begin{equation}\label{es0.3}
        \ell(\gz_{k,j}) = 1,
    \end{equation}
    for any $k\in\bN$ and $j\in\{1,\dots,2^k\}$.

    The tree-type structure is defined by
    $$\gz := \bigcup^{\fz}_{n = 1} \gz_n.$$

     \subsection{Fattening the tree-type core to fingers in order to form a Jordan domain \texorpdfstring{$\Oz$}{Lg}\label{ss2}}
    Based on the tree-like core bones $\gz$ in Section \ref{ss1}, we construct the Jordan domain $\Oz$.

To begin, we consider one branch at the time and fatten it to a finger. We take care that $h_{\Omega} \in L^1$ with the help of the distance function $g$ defined below.
    By the construction of $\gz$, for any sequence $\{j_k\}_{k\in\bN^+}$ such that $j_k\in\{1,\dots,2^k\}$ and $j_k = \lceil j_{k + 1}/2\rceil$,
    there is a unique smooth injective curve connecting all the points $\{r_{k,j_k}\}_{k\in\bN^+}$ starting from $p_0$.
    Let $\az\colon [0,\fz)\to \C$ be the parametrization of this smooth curve such that $\az(0) = p_0$ and $\ell(\az([0,t])) = t$.
    Thus, by \eqref{es0.3}, we have $\az(k) = r_{k,j_k}$.

    Consider the distance function $g(x) = e^{-M}\exp\left(-\exp(x)\right)$ where $M$ is a positive constant.
    For any $t\in [0,\fz)$, 
    there are two points $x_1(t)$ and $x_2(t)$ such that 
    \begin{align}
    &[x_1(t),x_2(t)] \perp \az\quad \text{and}\quad [x_1(t),x_2(t)] \cap \az = \{\az(t)\};\label{con-a}\\
    &|x_1(t) - \az(t)| = |x_2(t) - \az(t)| = g(t).\label{con-b}
    \end{align}
    We can choose $x_1(t)$ and $x_2(t)$ to be continuous with respect to $t\in [0,\fz)$ and $\Im x_1(0) > \Im x_2(0)$.

\begin{lemma}\label{claim1}
    The functions $x_1 \colon [0, \infty) \to \C$ and $x_2 \colon [0, \infty) \to \C$ are both injective, when $M$ is large enough.
\end{lemma}
    
\begin{proof}
    We first show that, for any $k\in\bN$, $x_1$ is injective on $[k - 1,k]$
    (on this interval $\az_J$ goes from one branch point $r_{k - 1,j_{k-1}}$ to another $r_{k,j_k}$).
    By \eqref{con-b} and the definition of $g$,
    $$|x_1(t) - \az(t)| = g(t) \leq g\left(k - 1\right) = e^{-M}\exp\left(-\exp(k - 1)\right).$$
    Since
    $$\frac{\exp\left(-\exp(k)\right)}{2^{-(2k + 6)}}\to 0\quad\text{as}\quad k\to\fz,$$
    we can choose $M$ large enough, so that, for all $k\in\bN^+$ and $t\in[k - 1,k]$,
\begin{equation}\label{es0.1}
    |x_1(t) - \az(t)| \leq\frac{1}{2^{2k + 6}}.
\end{equation}
Notice that the upper bound $1/(2^{2k + 6})$ is strictly smaller than  the radius of the tangent circle, i.e., $1/(2^{2k + 3})$, that we used in the construction of $\gamma$.

As $\az$ is injective and our distance function  $g$ is decreasing, we have $x_1$ is injective on $[k - 1,k]$.
By the construction of $\gz$, \eqref{es0.1} further shows that $x_1$ is injective on $[0,\fz)$.

    A similar argument shows the injectivity of $x_2$.
\end{proof}

    By Lemma \ref{claim1}, there is a domain bounded by $x_1$, $x_2$ and $[x_1(0),x_2(0)]$, denoted by $S$. 
    For any $x\in S$, we can write $x = (t,r)\in [0,\fz)\times [-g(t),g(t)]$ where
    $x\in [x_1(t),x_2(t)]$ and $|x_1(t) - x| = g(t) - r$.

\begin{lemma}\label{claim2}
    For any $x = (t,r)\in S$, when $M$ is large enough,
    $$d(x,\pd S) \sim \min\{|x_1(t) - x|,|x_2(t) - x|\} \sim \min\{g(t) - r, g(t) + r\}.$$
\end{lemma}
We prove this technical claim at the end of this paper, see Section \ref{ss4}.

\smallskip

Towards the integrability of hyperbolic distance, $h_{\Omega}\in L^1$,   decompose the domain $S$ into countably many parts $\{S_k\}^{\fz}_{k = 1}$ such that
    $S_k\subset S$ is bounded by $x_1([k - 1,k])$, $x_2([k - 1,k])$, $[x_1(k - 1), x_2(k - 1)]$ and $[x_1(k),x_2(k)]$.

Recall that the quasihyperbolic metric, for any $x, y \in S$, is
    $$k_{S}(x,y) = \inf_{\Gamma} \int_{\Gamma} \frac{1}{d(z,\pd S)} \daz,$$
where the infimum can be taken over all piecewise smooth curves $\Gamma$ in $S$ that join $x$ and $y$. By the definition of $\az$, we have $\ell(\az([0,t])) = t$.
    By Lemma \ref{claim2}, for any $x = (t,r)\in S_k$,
    \begin{align*}
        k_S(p_0,x) 
        &\leq k_S(p_0,\az(t)) + k_S(\az(t),x)\\
        &\leq \int_{\az([0,t])}\frac{1}{d(z,\pd S)}\,|\od z| + \int_{[\az(t),x]}\frac{1}{d(z,\pd S)}\,|\od z|\\
        &\lesssim \int^{t}_0\frac{1}{g(s)}\,\od s + \int^r_0 \frac{1}{g(t) - |s|}\,\od s\\
        &\lesssim \frac{1}{g(t)\exp(t)} + \log \frac{g(t)}{g(t) - |r|}.
    \end{align*}
    Thus
    \begin{equation}\label{es0.4}
\aligned
        \int_{S_k} k_S(p_0,x)\,\od x
        & = \int_{\az([k - 1,k])} \int^{g(t)}_{-g(t)} k_S(p_0, (t,r))\,\od r\od t\\
        & \lesssim \int_{\az([k - 1,k])} \frac{1}{\exp(t)} + \int^{g(t)}_{0}\log \frac{g(t)}{g(t) - |r|}\,\od r\od t\\
        & \lesssim \int_{\az([k - 1,k])} e^{- t} + g(t)\,\od t \lesssim e^{- k}.
\endaligned
\end{equation}

    Define the domain as the union of fingers, that is, $$\Oz := \bigcup S,$$
    where the union contains all the choices of $\{j_k\}^{\fz}_{k = 1}$ satisfying
    $j_k\in\{1,\dots,2^k\}$ and $j_k = \lceil j_{k + 1}/2\rceil$.
    By the definition of $S$, $\Oz$ is clearly a domain. 
    We show that $\Oz$ is a Jordan domain.
    
    First, we construct a continuous function $\vz\colon [-1,2] \to\pd\Oz$ by induction.
    For any sequence $J := \{j_k\}^{\fz}_{k = 1}$ satisfying
    $j_k\in\{1,\dots,2^k\}$ and $j_k = \lceil j_{k + 1}/2\rceil$,
    let $x_{1,J}$ and $x_{2,J}$ be the curve defined before.
    Set
    $$r_{k,j_k}^+ := x_{1,J}(k)\quad\text{and}\quad r_{k,j_k}^- := x_{2,J}(k).$$
    For simplification, consider the following schematic diagram of $\Oz$:

\smallskip
\tikzset{every picture/.style={line width=0.75pt}} 

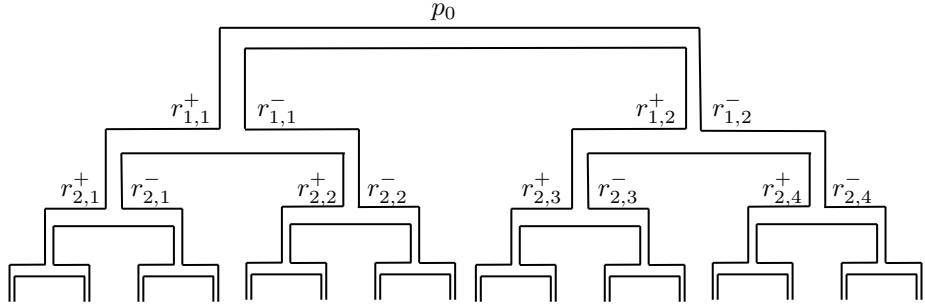
\begin{tikzpicture}[x=0.75pt,y=0.75pt,yscale=-1,xscale=1]

\draw    (210.89,103.44) -- (450.2,103.44) ;
\draw    (223.44,113.83) -- (443.56,113.44) ;
\draw    (210.89,103.44) -- (210.89,154.24) ;
\draw    (223.44,113.83) -- (223.44,154.24) ;
\draw    (154.01,154.44) -- (210.89,154.24) ;
\draw    (161.79,166.58) -- (273.28,166.58) ;
\draw    (154.01,154.44) -- (154.01,194.44) ;
\draw    (272.54,166.58) -- (272.54,193.24) ;
\draw    (280.32,154.44) -- (280.32,193.24) ;
\draw    (223.44,154.64) -- (280.32,154.44) ;
\draw    (127.9,203.24) -- (188.01,203.24) ;
\draw    (123.71,194.44) -- (123.71,222.57) ;
\draw    (127.9,203.24) -- (127.9,222.57) ;
\draw    (188.01,203.24) -- (188.01,222.57) ;
\draw    (192.21,194.44) -- (192.21,222.57) ;
\draw    (106,222.67) -- (123.94,222.57) ;
\draw    (108.45,228.48) -- (143.62,228.48) ;
\draw    (106,222.67) -- (106,241.24) ;
\draw    (108.45,228.48) -- (108.45,241.24) ;
\draw    (143.39,228.48) -- (143.39,241.24) ;
\draw    (145.84,222.67) -- (145.84,241.24) ;
\draw    (127.9,222.77) -- (145.84,222.67) ;
\draw    (170.31,222.67) -- (188.25,222.57) ;
\draw    (172.76,228.48) -- (207.93,228.48) ;
\draw    (170.31,222.67) -- (170.31,241.24) ;
\draw    (172.76,228.48) -- (172.76,241.24) ;
\draw    (207.69,228.48) -- (207.69,241.24) ;
\draw    (210.15,222.67) -- (210.15,241.24) ;
\draw    (192.21,222.77) -- (210.15,222.67) ;
\draw    (241.89,193.44) -- (272.54,193.24) ;
\draw    (246.08,202.24) -- (306.2,202.24) ;
\draw    (241.89,193.44) -- (241.89,221.57) ;
\draw    (246.08,202.24) -- (246.08,221.57) ;
\draw    (306.2,202.24) -- (306.2,221.57) ;
\draw    (310.39,193.44) -- (310.39,221.57) ;
\draw    (224.18,221.67) -- (242.12,221.57) ;
\draw    (226.64,227.48) -- (261.8,227.48) ;
\draw    (224.18,221.67) -- (224.18,240.24) ;
\draw    (226.64,227.48) -- (226.64,240.24) ;
\draw    (261.57,227.48) -- (261.57,240.24) ;
\draw    (264.02,221.67) -- (264.02,240.24) ;
\draw    (246.08,221.77) -- (264.02,221.67) ;
\draw    (288.49,221.67) -- (306.43,221.57) ;
\draw    (290.94,227.48) -- (326.11,227.48) ;
\draw    (288.49,221.67) -- (288.49,240.24) ;
\draw    (290.94,227.48) -- (290.94,240.24) ;
\draw    (325.88,227.48) -- (325.88,240.24) ;
\draw    (328.33,221.67) -- (328.33,240.24) ;
\draw    (310.39,221.77) -- (328.33,221.67) ;
\draw    (279.74,193.64) -- (310.39,193.44) ;
\draw    (123.71,194.44) -- (154.01,194.44) ;
\draw    (161.79,166.58) -- (162.14,194.44) ;
\draw    (162.14,194.44) -- (192.21,194.44) ;
\draw    (443.56,113.44) -- (443.56,154.24) ;
\draw    (450.2,103.44) -- (450.94,155.24) ;
\draw    (386.68,154.44) -- (443.56,154.24) ;
\draw    (394.46,166.58) -- (505.95,166.58) ;
\draw    (386.68,154.44) -- (386.68,194.44) ;
\draw    (505.21,166.58) -- (505.21,193.24) ;
\draw    (513,155.44) -- (512.99,193.24) ;
\draw    (450.94,155.24) -- (513,155.44) ;
\draw    (360.57,203.24) -- (420.68,203.24) ;
\draw    (356.38,194.44) -- (356.38,222.57) ;
\draw    (360.57,203.24) -- (360.57,222.57) ;
\draw    (420.68,203.24) -- (420.68,222.57) ;
\draw    (424.88,194.44) -- (424.88,222.57) ;
\draw    (338.67,222.67) -- (356.61,222.57) ;
\draw    (341.12,228.48) -- (376.29,228.48) ;
\draw    (338.67,222.67) -- (338.67,241.24) ;
\draw    (341.12,228.48) -- (341.12,241.24) ;
\draw    (376.06,228.48) -- (376.06,241.24) ;
\draw    (378.51,222.67) -- (378.51,241.24) ;
\draw    (360.57,222.77) -- (378.51,222.67) ;
\draw    (402.98,222.67) -- (420.92,222.57) ;
\draw    (405.43,228.48) -- (440.6,228.48) ;
\draw    (402.98,222.67) -- (402.98,241.24) ;
\draw    (405.43,228.48) -- (405.43,241.24) ;
\draw    (440.36,228.48) -- (440.36,241.24) ;
\draw    (442.82,222.67) -- (442.82,241.24) ;
\draw    (424.88,222.77) -- (442.82,222.67) ;
\draw    (474.56,193.44) -- (505.21,193.24) ;
\draw    (478.75,202.24) -- (538.87,202.24) ;
\draw    (474.56,193.44) -- (474.56,221.57) ;
\draw    (478.75,202.24) -- (478.75,221.57) ;
\draw    (538.87,202.24) -- (538.87,221.57) ;
\draw    (543.06,193.44) -- (543.06,221.57) ;
\draw    (456.85,221.67) -- (474.79,221.57) ;
\draw    (459.31,227.48) -- (494.47,227.48) ;
\draw    (456.85,221.67) -- (456.85,240.24) ;
\draw    (459.31,227.48) -- (459.31,240.24) ;
\draw    (494.24,227.48) -- (494.24,240.24) ;
\draw    (496.69,221.67) -- (496.69,240.24) ;
\draw    (478.75,221.77) -- (496.69,221.67) ;
\draw    (521.16,221.67) -- (539.1,221.57) ;
\draw    (523.61,227.48) -- (558.78,227.48) ;
\draw    (521.16,221.67) -- (521.16,240.24) ;
\draw    (523.61,227.48) -- (523.61,240.24) ;
\draw    (558.55,227.48) -- (558.55,240.24) ;
\draw    (561,221.67) -- (561,240.24) ;
\draw    (543.06,221.77) -- (561,221.67) ;
\draw    (512.41,193.64) -- (543.06,193.44) ;
\draw    (356.38,194.44) -- (386.68,194.44) ;
\draw    (394.46,166.58) -- (394.81,194.44) ;
\draw    (394.81,194.44) -- (424.88,194.44) ;

\draw (185,135) node [anchor=north west][inner sep=0.75pt]   [align=left] {$r_{1,1}^+$};
\draw (228,135) node [anchor=north west][inner sep=0.75pt]   [align=left] {$r_{1,1}^-$};
\draw (417,135) node [anchor=north west][inner sep=0.75pt]   [align=left] {$r_{1,2}^+$};
\draw (455,135) node [anchor=north west][inner sep=0.75pt]   [align=left] {$r_{1,2}^-$};
\draw (130,175) node [anchor=north west][inner sep=0.75pt]   [align=left] {$r_{2,1}^+$};
\draw (165,175) node [anchor=north west][inner sep=0.75pt]   [align=left] {$r_{2,1}^-$};
\draw (248,175) node [anchor=north west][inner sep=0.75pt]   [align=left] {$r_{2,2}^+$};
\draw (283,175) node [anchor=north west][inner sep=0.75pt]   [align=left] {$r_{2,2}^-$};
\draw (360,175) node [anchor=north west][inner sep=0.75pt]   [align=left] {$r_{2,3}^+$};
\draw (480,175) node [anchor=north west][inner sep=0.75pt]   [align=left] {$r_{2,4}^+$};
\draw (398,175) node [anchor=north west][inner sep=0.75pt]   [align=left] {$r_{2,3}^-$};
\draw (515,175) node [anchor=north west][inner sep=0.75pt]   [align=left] {$r_{2,4}^-$};
\draw (315,90) node [anchor=north west][inner sep=0.75pt]   [align=left] {$p_0$};
\end{tikzpicture}

\smallskip

    Let $l(a,b)\subset \pd\Oz$ be the curve connecting $a$ and $b$ with finite length,
    and $(R_{k,j} - r_k,R_{k,j} + r_k)$, $j\in\{1,\dots,2^{k - 1}\}$ denote the removed open interval in the $k$th step in the
    construction of the Smith--Volterra--Cantor set $C$ from left to right, and 
    $$r_k = \frac{4^{-k}}{2}.$$
    
\subsection*{Step 0.} Let $\vz(-1) = \vz(2) = p_0$.   
\subsection*{Step $1-1$.}
        \begin{align*}
            &\vz\colon [-1,-1/2]\to l(p_0,r_{1,1}^+)\quad\text{s.t.}\quad\vz(-1/2) = r_{1,1}^+;\\
            &\vz\colon [3/2,2]\to l(r_{1,2}^-,p_0)\quad\text{s.t.}\quad \vz(3/2) = r_{1,2}^-;\\
            &\vz\colon \left[R_{1,1} - \frac{r_1}{2},R_{1,1} + \frac{r_1}{2}\right]\to l(r_{1,1}^-,r_{1,2}^+)\\
            &\qquad\text{s.t.}\quad \vz\left(R_{1,1} - \frac{r_1}{2}\right) = r_{1,1}^-\quad\text{and}\quad\vz\left(R_{1,1} + \frac{r_1}{2}\right) = r_{1,2}^+,
        \end{align*}
        where $\vz$ is homeomorphism on every interval above with constant speed.
\subsection*{Step $K-1$ (for any $K \geq 2$).}
        \begin{align*}
            &\vz\colon [-2^{-(K - 1)}, -2^{-K}]\to l(r_{K - 1,1}^+,r_{K,1}^+)\quad\text{s.t.}\quad \vz(-2^{-K}) = r_{K,1}^+;\\
            &\vz\colon \left[R_{K,1} - \frac{r_K}{2}, R_{K,1} + \frac{r_K}{2}\right]\to l(r_{K,1}^-,r_{K,2}^+)\\
            &\qquad\text{s.t.}\quad \vz\left(R_{K,1} - \frac{r_K}{2}\right) = r_{K,1}^-\quad\text{and}\quad\vz\left(R_{K,1} + \frac{r_K}{2}\right) = r_{K,2}^+;\\
            &\vz\colon \left[R_{K - 1,1} - \frac{3r_{K - 1}}{4}, R_{K - 1,1} - \frac{r_{K - 1}}{2}\right]\to l(r_{K,2}^-,r_{K - 1,1}^-)\\
            &\qquad\text{s.t.}\quad \vz\left(R_{K - 1,1} - \frac{3r_{K - 1}}{4}\right) = r_{K,2}^-,
        \end{align*}       
        where $\vz$ is homeomorphism on every interval above with constant speed.
\subsection*{Step $K-2^{K - 1}$ (for any $K \geq 2$).}
        \begin{align*}
            &\vz\colon [1 + 2^{-K}, 1 + 2^{- (K - 1)}]\to l(r_{K,2^K}^-,r_{K - 1,2^{K - 1}}^-)\quad\text{s.t.}\quad \vz(1 + 2^{-K}) = r_{K,2^K}^-;\\
            &\vz\colon \left[R_{K,2^{K - 1}} - \frac{r_K}{2}, R_{K,2^{K - 1}} + \frac{r_K}{2}\right]\to l(r_{K,2^K - 1}^-,r_{K,2^K}^+)\\
            &\qquad\text{s.t.}\quad \vz\left(R_{K,2^{K -1}} - \frac{r_K}{2}\right) = r_{K,2^K - 1}^-\quad\text{and}\quad\vz\left(R_{K,2^{K - 1}} + \frac{r_K}{2}\right) = r_{K,2^K}^+;\\
            &\vz\colon \left[R_{K - 1,2^{K - 2}} + \frac{r_{K - 1}}{2}, R_{K - 1,2^{K - 2}} + \frac{3r_{K - 1}}{4}\right]\to l(r_{K - 1,2^{K - 1}}^+,r_{K,2^K - 1}^+)\\
            &\qquad\text{s.t.}\quad \vz\left(R_{K - 1,2^{K - 2}} + \frac{3r_{K - 1}}{4}\right) = r_{K,2^K - 1}^+,
        \end{align*}       
        where $\vz$ is homeomorphism on every interval above with constant speed.
\subsection*{Step $K-j$ (for any $K \geq 3$ and $2\leq j\leq 2^{K - 1} - 1$).} 
Let
        $$
        \aligned n_1(K,j) &:= \min\{n\in\bN:\ l(r_{n,m}^+,r_{K,2j - 1}^+)\ \text{exists for some }m\in\{1,\dots,2^n\}\},\\
        n_2(K,j) &:= \min\{n\in\bN:\ l(r_{n,m}^-,r_{K,2j}^-)\ \text{exists for some }m\in\{1,\dots,2^n\}\}.
        \endaligned$$
        Define the $\vz$ as
        \begin{align*}
            &\vz\colon \left[\vz^{-1}(r_{K - 1,j}^+) , \vz^{-1}(r_{K - 1,j}^+) + \frac{r_{n_1(K,j)}}{2^{K - n_1(K,j) + 1}}\right]\to l(r_{K - 1,j}^+,r_{K,2j - 1}^+)\\
            &\qquad\text{s.t.}\quad \vz\left(\vz^{-1}(r_{K - 1,j}^+) + \frac{r_{n_1(K,j)}}{2^{K - n_1(K,j) + 1}}\right) = r_{K,2j - 1}^+;\\
            &\vz\colon \left[R_{K,j} - \frac{r_K}{2}, R_{K,j} + \frac{r_K}{2}\right]\to l(r_{K,2j - 1}^-,r_{K,2j}^+)\\
            &\qquad\text{s.t.}\quad \vz\left(R_{K,j} - \frac{r_K}{2}\right) = r_{K,2j - 1}^-\quad\text{and}\quad\vz\left(R^K_j + \frac{r_K}{2}\right) = r_{K,2j}^+;\\
            &\vz\colon \left[\vz^{-1}(r_{K - 1,j}^-) - \frac{r_{n_2(K,j)}}{2^{K - n_2(K,j) + 1}}, \vz^{-1}(r_{K - 1,j}^-)\right]\to l(r_{K,2j}^-,r_{K - 1,j}^-)\\
            &\qquad\text{s.t.}\quad \vz\left(\vz^{-1}(r_{K - 1,j}^-) - \frac{r_{n_2(K,j)}}{2^{K - n_2(K,j) + 1}}\right) = r_{K,2j}^-,
        \end{align*} 
        where $\vz$ is homeomorphism on every interval above with constant speed, and $\vz^{-1}$ is defined in Step $(K - 1)-(\lceil j\rceil)$.

    Especially, define $\vz(x) = (x, 0)$ 
    for any $x\in C$. By induction, we obtain a continuous map
    $$\vz\colon [-1,2]\to \pd\Oz\quad\text{with}\quad \vz(-1) = \vz(2),$$
    so that the restriction of $\vz$ to $(-1,2)$ is injective and
    \begin{itemize}
        \item[(1)] $\vz(-\frac{1}{2^k}) = r_{k,1}^+$ for any $k\in\bN^+$;
        \item[(2)] $\vz(1 + \frac{1}{2^k}) = r_{k,2^k}^-$ for any $k\in\bN^+$;
        \item[(3)] for any $k\in\bN^+$, $j\in\{1,2,\dots,2^{k - 1}\}$ and $n\in\bN^+$,
        \begin{align*}
            &\vz\left(R_{k,j} - (1 - \frac{1}{2^n})r_k\right) = r_{k + n - 1, 2^{n - 1}(2j - 1)}^-\\
            &\vz\left(R_{k,j} + (1 - \frac{1}{2^n})r_k\right) = r_{k + n - 1, 2^{n - 1}(2j - 1) + 1}^+.
        \end{align*}
    \end{itemize}
    After a  bi-Lipschitz reparametrization, we may assume that 
    $$\vz\colon \pd \D\to\pd\Oz,$$
    with $\vz(C')=C$ for a Cantor-type set $C'$ of positive length.
    By the construction and the definition of $\Oz$, 
    $\vz$ is continuous and bijective, which shows that $\vz$ is a homeomorphism.
    Then, the domain $\Oz$ is a Jordan domain.

    \subsection{Counterexample, Proof of  Theorem \ref{counterex}\label{ss3}} 

 Our $\vz$ and $\Oz$ show the optimality of Theorem \ref{The1}, i.e., when $q = 1$, the claim of the theorem does not hold. 

\begin{proof}[Proof of  Theorem \ref{counterex}]
    By the definition of $\Oz$ and \eqref{es0.4}, we conclude that
    \begin{align*}
        \int_{\Oz} k_{\Oz}(p_0,z)\,\od z
        &\lesssim \sum^{\fz}_{k = 1} 2^k\int_{S_k}k_{S}(p_0,z)\,\od z\\
        &\lesssim \sum^{\fz}_{k = 1}2^k e^{-k}<\fz,
    \end{align*}
    where $S$ and $S_k$ are defined as in Section \ref{ss2}. As the hyperbolic distance $h_{\Omega}$ is comparable to the quasihyperbolic distance in our Jordan domain $\Omega,$ we have that $h_{\Omega} \in L^1$.

Let $\Phi\colon \overline{\D}\to\overline{\Oz}$  be a homeomorphism such that $\Phi(z) = \vz(z)$ for all $z\in\pd\D$.  Let $\gamma_z$ be any curve connecting an interior point of the unit disk  to a point $z \in C'=\phi^{-1}(C)$; recall that $C' \subset \pd\D$ is of positive length. The points of  $C$ are at the ends of infinite branches, by our definition of $\Omega$. Hence the image curve $\Phi(\gamma_z)$ needs to travel through infinitely many ``snakes" whose cores are formed by the $\gamma_{k,j}$ constructed in Section \ref{ss1}. Now, by \eqref{es0.3}, the length of the core of the branch one needs to travel through is infinite. Thus
$$
\ell(\Phi(\gamma_z)) = \infty.
$$
Especially, by considering suitable radial segments we conclude that $\Phi \not\in W^{1,1}(\D,\C)$.
\end{proof}

\subsection{Proof of Lemma \ref{claim2}}\label{ss4}
\begin{proof}[Proof of Lemma \ref{claim2}]
    For any $(s_0,t_0) \in R := \{(s,t)\in\C:\ |t|\leq g(s),\}$ with $s_0 \geq 0$,
    it is easy to check that $g'(s_0)\leq 0$, $g''(s_0)\geq 0$ and $|g'(s_0)|$ is uniformly bounded. 
    By the proof of \cite[Theorem 3.19]{Staples:89}, we have
    \begin{equation}\label{es0.2}
        d((s_0,t_0),\pd R) \sim g(s_0) - |t_0|. 
    \end{equation}

    Now we prove the claim. 
    For simplicity, assume that $x_1$, $x_2$ and $\az$ are in the position showed in the following figure.

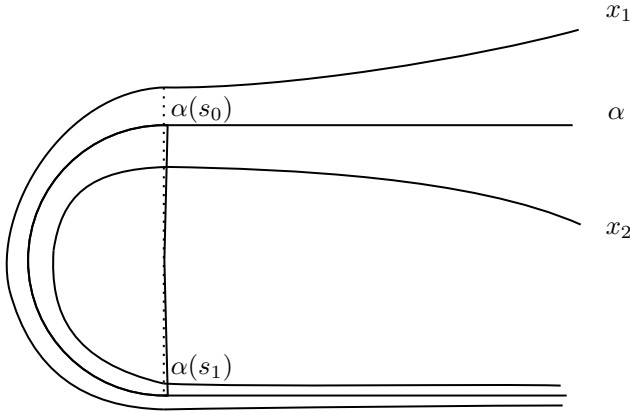
\tikzset{every picture/.style={line width=0.75pt}} 

\begin{tikzpicture}[x=0.75pt,y=0.75pt,yscale=-1,xscale=1]

\draw  [draw opacity=0] (249.98,232.42) .. controls (249.41,232.43) and (248.85,232.44) .. (248.28,232.44) .. controls (210.73,232.44) and (180.28,201.99) .. (180.28,164.44) .. controls (180.28,126.88) and (210.73,96.44) .. (248.28,96.44) .. controls (248.82,96.44) and (249.37,96.44) .. (249.91,96.46) -- (248.28,164.44) -- cycle ; \draw   (249.98,232.42) .. controls (249.41,232.43) and (248.85,232.44) .. (248.28,232.44) .. controls (210.73,232.44) and (180.28,201.99) .. (180.28,164.44) .. controls (180.28,126.88) and (210.73,96.44) .. (248.28,96.44) .. controls (248.82,96.44) and (249.37,96.44) .. (249.91,96.46) ;  
\draw    (248,96.44) -- (452,96.44) ;
\draw    (248,232.44) -- (449,232.44) ;
\draw    (248,77.44) .. controls (308,78.44) and (404,62.44) .. (455,48.44) ;
\draw  [dash pattern={on 0.84pt off 2.51pt}]  (248,77.44) -- (248,232.44) ;
\draw    (171,179.44) .. controls (162,139.44) and (198,79.44) .. (248,77.44) ;
\draw    (248,239.44) .. controls (208,238.44) and (182,220.44) .. (171,179.44) ;
\draw    (248,239.44) .. controls (298,238.44) and (417,237.44) .. (447,237.44) ;
\draw    (250,117.44) .. controls (310,118.44) and (402,122.44) .. (456,146.44) ;
\draw    (193,159.44) .. controls (197,132.44) and (210,118.44) .. (250,117.44) ;
\draw    (248,226.44) .. controls (202,215.44) and (191,191.44) .. (193,159.44) ;
\draw    (248,226.44) .. controls (295,228.44) and (413,226.44) .. (446,227.44) ;

\draw (467,35) node [anchor=north west][inner sep=0.75pt]   [align=left] {$x_1$};
\draw (468,86) node [anchor=north west][inner sep=0.75pt]   [align=left] {$\az$};
\draw (467,144) node [anchor=north west][inner sep=0.75pt]   [align=left] {$x_2$};
\draw (250,80.44) node [anchor=north west][inner sep=0.75pt]   [align=left] {$\az(s_0)$};
\draw (250,210) node [anchor=north west][inner sep=0.75pt]   [align=left] {$\az(s_1)$};
\end{tikzpicture}

Recall that we write $x = (t,r)\in S$ with $(t,r)\in [0,\fz)\times [-g(t),g(t)]$ if
$$[x_1(t),x_2(t)] \perp \az,\quad [x_1(t),x_2(t)] \cap \az = \{\az(t)\},$$
$$|x_1(t) - \az(t)| = |x_2(t) - \az(t)| = g(t),$$
and
$$x\in [x_1(t),x_2(t)],\quad |x_1(t) - x| = g(t) - r.$$
Moreover, in the picture above, $\az$ can be decomposed into two horizontal line segments
and a half circle $\az([s_0,s_1])$.

Let $k\in\bN$ such that $t\in [k - 1,k]$. Then, by \eqref{es0.1},
\begin{equation}\label{es0.6}
    |g(t)| \leq \frac{1}{2^{2k + 6}} = \frac{|\az(s_0) - \az(s_1)|}{32}.
\end{equation}
    \begin{itemize}
        \item[(a)]If $|x_1(t) - x|\geq \frac{3g(t)}{2}$, then 
    $d(x,x_2) \leq d(x,\az) \leq d(x,x_1)$. By \eqref{es0.2}, we have
    $$d(x,\pd S) = d(x,x_2) \gtrsim |x_2(t) - x|\sim \min\{|x_1(t) - x|,|x_2(t) - x|\}.$$
        \item[(b)]If $|x_1(t) - x|\leq \frac{3g(t)}{2}$, 
    then 
    $$\min\{|x_1(t) - x|,|x_2(t) - x|\}\sim |x_1(t) - x|\lesssim d(x,x_2).$$ 
    \end{itemize}
    Thus, it suffices to show 
    \begin{equation}\label{es0.5}
        d(x,x_1) \sim |x_1(t) - x|.
    \end{equation}
    By \eqref{es0.6}, without loss of generality, we only need to consider the situation when $x = (t,0) = \az(t)$.
    
\subsection*{Case I} $x = (t,0)$ is on the horizontal part of $\az$ (see the figure below).
        \tikzset{every picture/.style={line width=0.75pt}} 

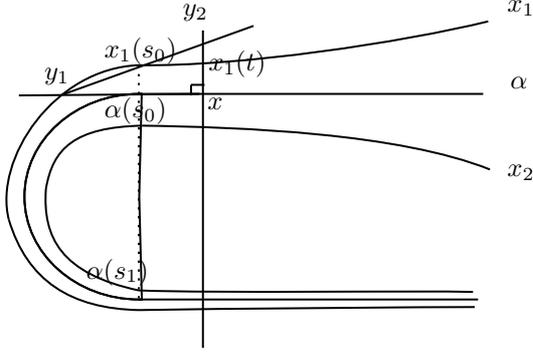
        \tikzset{every picture/.style={line width=0.75pt}} 

        \begin{tikzpicture}[x=0.75pt,y=0.75pt,yscale=-1,xscale=1]
        
        \draw  [draw opacity=0] (250.17,228.09) .. controls (249.74,228.1) and (249.31,228.1) .. (248.88,228.1) .. controls (217.24,228.1) and (191.6,204.95) .. (191.6,176.39) .. controls (191.6,147.83) and (217.24,124.68) .. (248.88,124.68) .. controls (249.29,124.68) and (249.7,124.68) .. (250.11,124.69) -- (248.88,176.39) -- cycle ; \draw   (250.17,228.09) .. controls (249.74,228.1) and (249.31,228.1) .. (248.88,228.1) .. controls (217.24,228.1) and (191.6,204.95) .. (191.6,176.39) .. controls (191.6,147.83) and (217.24,124.68) .. (248.88,124.68) .. controls (249.29,124.68) and (249.7,124.68) .. (250.11,124.69) ;  
        \draw    (248.64,124.68) -- (420.47,124.68) ;
        \draw    (248.64,228.1) -- (417.94,228.1) ;
        \draw    (248.64,110.23) .. controls (299.18,110.99) and (380.04,98.82) .. (422.99,88.18) ;
        \draw  [dash pattern={on 0.84pt off 2.51pt}]  (248.64,110.23) -- (248.64,228.1) ;
        \draw    (183.78,187.8) .. controls (176.2,157.38) and (206.52,111.75) .. (248.64,110.23) ;
        \draw    (248.64,233.43) .. controls (214.95,232.67) and (193.05,218.98) .. (183.78,187.8) ;
        \draw    (248.64,233.43) .. controls (290.75,232.67) and (390.99,231.9) .. (416.25,231.9) ;
        \draw    (250.32,140.65) .. controls (300.86,141.41) and (378.35,144.45) .. (423.84,162.7) ;
        \draw    (202.31,172.59) .. controls (205.68,152.06) and (216.63,141.41) .. (250.32,140.65) ;
        \draw    (248.64,223.54) .. controls (209.89,215.17) and (200.63,196.92) .. (202.31,172.59) ;
        \draw    (248.64,223.54) .. controls (288.23,225.06) and (387.62,223.54) .. (415.41,224.3) ;
        \draw    (280.65,92.74) -- (280.65,252.44) ;
        \draw    (188.84,125.44) -- (278.96,124.68) ;
        \draw    (211.58,124.68) -- (305.91,90.46) ;
        \draw    (274.75,120.12) -- (281.49,120.12) ;
        \draw    (274.75,120.12) -- (274.75,125.44) ;
        
        \draw (431.13,75.92) node [anchor=north west][inner sep=0.75pt]   [align=left] {$x_1$};
        \draw (432.52,114.71) node [anchor=north west][inner sep=0.75pt]   [align=left] {$\az$};
        \draw (431.13,158.81) node [anchor=north west][inner sep=0.75pt]   [align=left] {$x_2$};
        \draw (282,101.51) node [anchor=north west][inner sep=0.75pt]   [align=left] {$x_1(t)$};
        \draw (281.46,125) node [anchor=north west][inner sep=0.75pt]   [align=left] {$x$};
        \draw (200,110) node [anchor=north west][inner sep=0.75pt]    {$y_1$};
        \draw (269.12,78.2) node [anchor=north west][inner sep=0.75pt]   [align=left] {$y_2$};
        \draw (230,125) node [anchor=north west][inner sep=0.75pt]   [align=left] {$\az(s_0)$};
        \draw (221,206) node [anchor=north west][inner sep=0.75pt]   [align=left] {$\az(s_1)$};
        \draw (230,95) node [anchor=north west][inner sep=0.75pt]   [align=left] {$x_1(s_0)$};
        
        \end{tikzpicture}

        Here $y_1$ is the intersection of the line $[x,\az(s_0)]$ and the curve $x_1([s_0,s_1])$,
        and $y_2$ is the intersection of two lines $[x,x_1(t)]$ and $[y_1,x_1(s_0)]$.

        By \eqref{es0.6}, we have $\angle x,y_1,y_2 \leq \frac{\pi}{4}$, which leads that 
        $$d(x,x_1([s_0,s_1]))\geq \frac{|y_2 - x|}{\sqrt{2}} \geq \frac{|x_1(t) - x|}{\sqrt{2}}.$$
        Combining with \eqref{es0.2}, we conclude that
        $$d(x,x_1) \gtrsim |x_1(t) - x|.$$
        The proof is similar when $t\geq s_1$.

\subsection*{Case II} $x = (t,0)$ is such that $t\in [s_0,s_1]$.
        \tikzset{every picture/.style={line width=0.75pt}} 

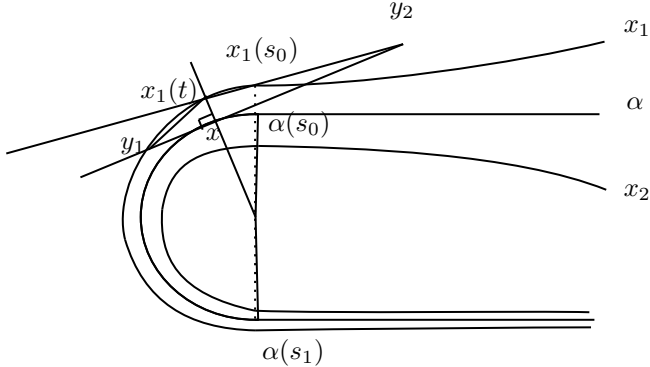
\begin{tikzpicture}[x=0.75pt,y=0.75pt,yscale=-1,xscale=1]

\draw  [draw opacity=0] (250.17,228.09) .. controls (249.74,228.1) and (249.31,228.1) .. (248.88,228.1) .. controls (217.24,228.1) and (191.6,204.95) .. (191.6,176.39) .. controls (191.6,147.83) and (217.24,124.68) .. (248.88,124.68) .. controls (249.29,124.68) and (249.7,124.68) .. (250.11,124.69) -- (248.88,176.39) -- cycle ; \draw   (250.17,228.09) .. controls (249.74,228.1) and (249.31,228.1) .. (248.88,228.1) .. controls (217.24,228.1) and (191.6,204.95) .. (191.6,176.39) .. controls (191.6,147.83) and (217.24,124.68) .. (248.88,124.68) .. controls (249.29,124.68) and (249.7,124.68) .. (250.11,124.69) ;  
\draw    (248.64,124.68) -- (420.47,124.68) ;
\draw    (248.64,228.1) -- (417.94,228.1) ;
\draw    (248.64,110.23) .. controls (299.18,110.99) and (380.04,98.82) .. (422.99,88.18) ;
\draw  [dash pattern={on 0.84pt off 2.51pt}]  (248.64,110.23) -- (248.64,228.1) ;
\draw    (183.78,187.8) .. controls (176.2,157.38) and (206.52,111.75) .. (248.64,110.23) ;
\draw    (248.64,233.43) .. controls (214.95,232.67) and (193.05,218.98) .. (183.78,187.8) ;
\draw    (248.64,233.43) .. controls (290.75,232.67) and (390.99,231.9) .. (416.25,231.9) ;
\draw    (250.32,140.65) .. controls (300.86,141.41) and (378.35,144.45) .. (423.84,162.7) ;
\draw    (202.31,172.59) .. controls (205.68,152.06) and (216.63,141.41) .. (250.32,140.65) ;
\draw    (248.64,223.54) .. controls (209.89,215.17) and (200.63,196.92) .. (202.31,172.59) ;
\draw    (248.64,223.54) .. controls (288.23,225.06) and (387.62,223.54) .. (415.41,224.3) ;
\draw    (322.5,89.44) -- (161.5,156.44) ;
\draw    (216.5,98.44) -- (248.88,176.39) ;
\draw    (195.5,142.44) -- (224.5,115.44) ;
\draw    (220.5,127.44) -- (227.5,124.44) ;
\draw    (220.5,127.44) -- (222.5,132.44) ;
\draw    (322.5,89.44) -- (124.5,144.44) ;

\draw (431.13,75.92) node [anchor=north west][inner sep=0.75pt]   [align=left] {$x_1$};
\draw (432.52,114.71) node [anchor=north west][inner sep=0.75pt]   [align=left] {$\az$};
\draw (431.13,158.81) node [anchor=north west][inner sep=0.75pt]   [align=left] {$x_2$};
\draw (190,105) node [anchor=north west][inner sep=0.75pt]   [align=left] {$x_1(t)$};
\draw (222.5,130.44) node [anchor=north west][inner sep=0.75pt]   [align=left] {$x$};
\draw (180,135) node [anchor=north west][inner sep=0.75pt]    {$y_1$};
\draw (314.12,67.2) node [anchor=north west][inner sep=0.75pt]   [align=left] {$y_2$};
\draw (253.64,121.68) node [anchor=north west][inner sep=0.75pt]   [align=left] {$\az(s_0$)};
\draw (250.64,236.43) node [anchor=north west][inner sep=0.75pt]   [align=left] {$\az(s_1)$};
\draw (232.62,84.51) node [anchor=north west][inner sep=0.75pt]   [align=left] {$x_1(s_0)$};

\end{tikzpicture}

Here $[y_1,y_2]$ is on the tangent of $\az([s_0,s_1])$ passing through $x$,
$y_1\in x_1([s_0,s_1])$, and $y_2$ is the intersection of the tangent and the line $[x_1(t),x_1(s_0)]$.

By \eqref{es0.6}, it is clear that $\angle x,y_1,x_1(t) \leq \frac{\pi}{4}$ and $\angle x,y_2,x_1(t) \leq \frac{\pi}{4}$, thus
\begin{equation}\label{es0.7}
    d(x,x_1([s_0,s_1])) \gtrsim |x_1(t) - s|.
\end{equation}
On the other hand,
\tikzset{every picture/.style={line width=0.75pt}} 

\begin{tikzpicture}[x=0.75pt,y=0.75pt,yscale=-1,xscale=1]

\draw  [draw opacity=0] (250.17,228.09) .. controls (249.74,228.1) and (249.31,228.1) .. (248.88,228.1) .. controls (217.24,228.1) and (191.6,204.95) .. (191.6,176.39) .. controls (191.6,147.83) and (217.24,124.68) .. (248.88,124.68) .. controls (249.29,124.68) and (249.7,124.68) .. (250.11,124.69) -- (248.88,176.39) -- cycle ; \draw   (250.17,228.09) .. controls (249.74,228.1) and (249.31,228.1) .. (248.88,228.1) .. controls (217.24,228.1) and (191.6,204.95) .. (191.6,176.39) .. controls (191.6,147.83) and (217.24,124.68) .. (248.88,124.68) .. controls (249.29,124.68) and (249.7,124.68) .. (250.11,124.69) ;  
\draw    (248.64,124.68) -- (420.47,124.68) ;
\draw    (248.64,228.1) -- (417.94,228.1) ;
\draw    (248.64,110.23) .. controls (299.18,110.99) and (380.04,98.82) .. (422.99,88.18) ;
\draw  [dash pattern={on 0.84pt off 2.51pt}]  (248.64,110.23) -- (248.64,228.1) ;
\draw    (183.78,187.8) .. controls (176.2,157.38) and (206.52,111.75) .. (248.64,110.23) ;
\draw    (248.64,233.43) .. controls (214.95,232.67) and (193.05,218.98) .. (183.78,187.8) ;
\draw    (248.64,233.43) .. controls (290.75,232.67) and (390.99,231.9) .. (416.25,231.9) ;
\draw    (250.32,140.65) .. controls (300.86,141.41) and (378.35,144.45) .. (423.84,162.7) ;
\draw    (202.31,172.59) .. controls (205.68,152.06) and (216.63,141.41) .. (250.32,140.65) ;
\draw    (248.64,223.54) .. controls (209.89,215.17) and (200.63,196.92) .. (202.31,172.59) ;
\draw    (248.64,223.54) .. controls (288.23,225.06) and (387.62,223.54) .. (415.41,224.3) ;
\draw    (221.5,73.44) -- (222.5,179.44) ;
\draw  [dash pattern={on 0.84pt off 2.51pt}]  (98.5,125.44) -- (323.5,124.44) ;
\draw  [dash pattern={on 0.84pt off 2.51pt}]  (106.5,117.44) .. controls (162.5,116.44) and (191.5,117.44) .. (248.64,110.23) ;

\draw (431.13,75.92) node [anchor=north west][inner sep=0.75pt]   [align=left] {$x_1$};
\draw (432.52,114.71) node [anchor=north west][inner sep=0.75pt]   [align=left] {$\az$};
\draw (431.13,158.81) node [anchor=north west][inner sep=0.75pt]   [align=left] {$x_2$};
\draw (222,133.44) node [anchor=north west][inner sep=0.75pt]   [align=left] {$x$};
\draw (250.64,236.43) node [anchor=north west][inner sep=0.75pt]   [align=left] {$\az(s_1)$};
\draw (76.13,98.92) node [anchor=north west][inner sep=0.75pt]   [align=left] {$\wt{x}_1$};
\draw (205,95) node [anchor=north west][inner sep=0.75pt]   [align=left] {$\wt{y}_1$};
\draw (76.52,123.71) node [anchor=north west][inner sep=0.75pt]   [align=left] {$\wt{\az}$};
\draw (222,115) node [anchor=north west][inner sep=0.75pt]   [align=left] {$\wt{y}_2$};
\draw (253.64,121.68) node [anchor=north west][inner sep=0.75pt]   [align=left] {$\az(s_0$)};

\end{tikzpicture}

Let $\wt{\az}$ be the horizontal line through $\az(s_0)$, and $\wt{x}_1$ is the curve constructed as $x_1$ but with the line $\wt{\az}$.
$[x,\wt{y}_1]$ is perpendicular to $\wt{\az}$ and intersects with $\wt{x}_1$ at point $\wt{y}_1$, $[x,\wt{y}_1]\cap \az([s_0,s_1]) = \wt{y}_2$.
Then, by \eqref{es0.2} and \eqref{es0.7}, we have
\begin{align*}
    d(x,x_1([0,s_0])) \geq d(x,\wt{x}_1) \gtrsim |x - \wt{y}_1| \gtrsim |x - \wt{y}_2| \gtrsim |x_1(t) - x|.
\end{align*}
Combined with \eqref{es0.7}, the claim holds true in this situation.

If $x$ is close to $\az(s_1)$, the proof is just the same.

If the tagent through $x$ has two intersection points with $\az([s_0,s_1])$, the proof is similar and we omit the details here.

We have showed \eqref{es0.5}, and thus also our claim.
\end{proof}


\begin{thebibliography}{99}

\bibitem{Antman1995}
S. S. Antman, Nonlinear problems of elasticity, Applied Mathematical Sciences,
vol. 107, Springer-Verlag, New York, 1995.

    \bibitem{AIM2009}
    K. Astala, T. Iwaniec, G. Martin,
    Elliptic Partial Differential Equations and Quasiconformal Mappings in the Plane,
    Princeton University Press, 2009.

\bibitem{Ball1976}
J. M. Ball, Convexity conditions and existence theorems in nonlinear elasticity,
Arch. Rational Mech. Anal. 63 (1976/77), no. 4, 337--403.

\bibitem{Ciarlet1988}
P. G. Ciarlet, Mathematical elasticity Vol. I. Three-dimensional elasticity, Studies in
Mathematics and its Applications, vol. 20. North-Holland Publishing Co., Amsterdam, 1988.

\bibitem {HenclKoskela2014}
S. Hencl, P. Koskela, Lectures on mappings of finite distortion. Lecture Notes in Mathematics, 2096. Springer, Cham, 2014.

\bibitem{IwaniecMartin2001}
T. Iwaniec, G. Martin, Geometric Function Theory and Non-linear Analysis,
Oxford Mathematical Monographs, Oxford University Press, 2001.

\bibitem{JM}
P. W. Jones, N. G. Makarov, Density properties of harmonic measure,  Ann. of Math. (2), 142 (1995), no. 3, 427-455.

\bibitem{KoKoOn20}
P. Koskela, A. Koski, J. Onninen, Sobolev homeomorphic extensions onto John domains, J. Funct. Anal. 279 (2020), no. 10, 108719, 17 pp.
    \bibitem{KoskiOnninen2021}
A. Koski, J. Onninen, Sobolev homeomorphic extensions, J. Eur. Math. Soc. 23 (2021), no. 12, 4065–4089.

    \bibitem{KO2023}
    A. Koski, J. Onninen, 
    The Sobolev Jordan--Sch\"onflies problem, 
    Adv. Math.
    413 (2023), 108795

\bibitem{KR}
P. Koskela, S. Rohde, Hausdorff dimension and mean porosity, Math.Ann. 309 (1997), no. 4, 593--609.

\bibitem{Reshetnyak1989}
Yu. G. Reshetnyak, Space mappings with bounded distortion, American Mathematical
Society, Providence, RI, 1989.

    \bibitem{Staples:89} 
         S.~Staples, 
         $L^p$-averaging domains and the Poincar{\'e} inequality, 
         Ann. Acad. Sci. Fenn. Ser. A. I. Math. 14 (1989), 103--127.


\bibitem{SS}
S. Smith, D. A. Stegenga, Exponential integrability of the quasihyperbolic metric on H\"older domains, Ann. Acad. Sci. Fenn. Ser. A I  Math. 16 (1991), no. 2, 345--360.

\bibitem{Verchota2007}
G. C. Verchota, Harmonic homeomorphisms of the closed disc to itself need be in
$W^{1,p}$, $p < 2$, but not $W^{1,2}$, Proc. Amer. Math. Soc. (2007), vol. 135, no. 3, 891--894.

\bibitem{Zhang2019}
Y. R.-Y. Zhang, Schoenflies solutions with conformal boundary values may fail to be
Sobolev, Ann. Acad. Sci. Fenn. Ser. A I Math. (2019), vol. 44, no. 2, 791--796.
    

    \end{thebibliography}
\end{document}